\documentclass[a4paper]{IEEEtran}

\usepackage{graphicx}
\usepackage[latin1]{inputenc}
\usepackage{epstopdf}
\usepackage[cmex10]{amsmath}
\usepackage{amsfonts}
\usepackage{amssymb}
\usepackage{mathrsfs}
\usepackage{verbatim}
\usepackage{array}
\usepackage{dcolumn}
\usepackage{color}
\usepackage{arydshln}
\usepackage{url}
\usepackage{mathrsfs}
\usepackage{algorithm, algorithmic}
\usepackage{relsize}
\usepackage{changepage}
\usepackage[bb=boondox]{mathalfa}
\usepackage[dvipsnames]{xcolor}
\usepackage{subfigure}
\usepackage{scalerel}
\usepackage{threeparttable}
\usepackage{tabularx}
\usepackage{multirow}
\usepackage{multicol}
\usepackage{makecell}
\usepackage{flushend}
\usepackage{threeparttable}
\usepackage{booktabs}

\usepackage{derivative}
\newcommand{\pd}[2]{\pdv{#1}{#2}}

\usepackage{xurl} 
\usepackage{hyperref}
\hypersetup{
    breaklinks=true,           
    colorlinks=false,          
    citebordercolor={0 1 0},   
    linkbordercolor={1 0 0},   
    urlbordercolor={0 0 1}     
}

\usepackage{geometry}
\allowdisplaybreaks
\usepackage{amssymb}
\usepackage{nomencl}
\makenomenclature
\usepackage{etoolbox}
\renewcommand\nomgroup[1]{
  \item[\bfseries
  \ifstrequal{#1}{A}{Sets}{
  \ifstrequal{#1}{B}{Parameters}{
  \ifstrequal{#1}{C}{Decision Variables}{
  }}}
]}

\DeclareGraphicsExtensions{.eps}
\AppendGraphicsExtensions{.pdf}

\newtheorem{thm}{Theorem}

\newtheorem{prop}[thm]{Proposition}

\begin{document}

\title{An Alternative Settlement Rule for State-of-Charge-Limited Resources in Electricity Markets}

\author{

Alexandre Moreira, \IEEEmembership{Senior Member,~IEEE}

\thanks{A. Moreira is with the Lawrence Berkeley National Laboratory, Berkeley, CA, USA (e-mail: AMoreira@lbl.gov).}


}
 
\maketitle

\begin{abstract}
    Standard two settlement markets compute the total profits of their participants as a sum of day-ahead (DA) profits and incremental real-time (RT) profits.
    This procedure was originally designed for conventional generators, which can always deliver their day-ahead (DA) awarded schedule in real time because their capacity is invariant across market clearings.
    Energy storage resources, by contrast, are state-of-charge (SOC) limited. Therefore, their ability to honor a DA award in real time depends on the SOC available at the moment of dispatch,
    a quantity that can differ materially from what was assumed at the DA stage. 
    According to the current rules in CAISO, whenever the storage state of charge (SOC) is not sufficient to support the day-ahead award the RT market may impose an instruction to ``buy back'' or ``sell back''
    that may lead to a monetary loss that is covered by the system operator via uplifts in the bid cost recovery (BCR) scheme. These circumstances have resulted in storage resources capturing a disproportionate
    share of uplift payments---totaling approximately \$75.2 million in CAISO between 2022 and 2025. 
    This paper formalizes the asymmetry
    between SOC-unlimited (generator) and SOC-limited (storage) resources and proposes an alternative settlement rule in which the total storage profit is computed as the maximum between real-time only compensation and the existing day ahead plus incremental real-time compensation. Through our analysis and a case study based on a realistic version of the California system, we demonstrate that while the proposed rule does not guarantee non-negative profits in all cases, it can reduce---and in favorable
    conditions eliminate---the negative profit outcomes that trigger
    make-whole BCR payments.
\end{abstract}

\vspace{-0cm}
\section*{Nomenclature}\label{Nomenclature}

\subsection*{Sets}
\begin{description} [\IEEEsetlabelwidth{5000000}\IEEEusemathlabelsep]
    \item[$G_{n}$] Set of generators at bus $n$. 
    \item [$K_{n}$] Set of storage systems at bus $n$. 
\end{description}

\subsection*{Parameters}
\begin{description} [\IEEEsetlabelwidth{5000000}\IEEEusemathlabelsep]
    \item[$\lambda^{DA}_{n}$] Day-ahead energy price at bus $n$.
    \item[$\lambda^{RT}_{n}$] Real-time energy price at bus $n$.
    \item[${C}^{G}_{g}$] Generation cost of power plant $g$.
    \item[$\overline{P}^{G}_{g}$] Generation capacity of power plant $g$.
    \item[$v^{DA,\dagger}_{k,0}$] Available state of charge storage $k$ in day ahead.
    \item[$v^{RT,\dagger}_{k,0}$] Available state of charge storage $k$ in real time.
    \item[$\overline{V}_{k}$] State of charge limit of storage $k$.
\end{description}

\subsection*{Decision Variables}
\begin{description} [\IEEEsetlabelwidth{5000000}\IEEEusemathlabelsep]
    \item[$p^{G,DA}_g$] Generation output of generator $g$ in day ahead.
    \item[$p^{G,RT}_g$] Generation output of generator $g$ in real time.
    \item[$p^{ch,DA}_k$] Charging output of storage $k$ in day ahead.
    \item[$p^{ch,RT}_k$] Charging output of storage $k$ in real time.
    \item[$p^{dis,DA}_k$] Discharging output of storage $k$ in day ahead.
    \item[$p^{dis,RT}_k$] Discharging output of storage $k$ in real time.
    \item[$v^{DA}_k$] Resulting state of charge storage $k$ in day ahead after charging/discharging decision.
    \item[$v^{RT}_k$] Resulting state of charge storage $k$ in real time after charging/discharging decision.
    \item[$TP^{G}_g$] Total profit of generator $g$ considering day ahead and real time market outcomes.
    \item[$TP^{St}_k$] Total profit of storage $k$ considering day ahead and real time market outcomes.
\end{description}

\printnomenclature
\vspace{-0cm}
\section{Introduction}
\label{sec:intro}

\IEEEPARstart{G}{rid}-scale storage has grown from a negligible share of the resource mix to a central pillar of modern power systems. In California, installed storage capacity grew from less than 700 MW in 2019 to more than 21,000 MW in 2026~\cite{CEC_2026_survey_storage_installed_cap}. 
As storage assets proliferate in power grids, there is a growing concern that current market compensation rules (originally conceived for conventional generators) are not properly suited to energy-limited resources \cite{FERC_2025_order_accepting_bcr_changes}.

The underlying difficulty is related to legacy market design. Most current wholesale market frameworks were conceived for \emph{memoryless} assets, i.e., traditional generators whose operational state in one hour does not inherently impact their capability to serve the system in the next hours. Energy storage resources (ESRs), by contrast, are completely dependent on their state of charge (SOC), which imposes a temporal link that transforms every dispatch decision into a trade-off between immediate revenue and future flexibility. As highlighted by \cite{Zhao_2020_Multi-PeriodMarketDesign}, the fundamental challenge arises when multi-period constraints, such as the SOC balance, are superimposed onto markets that essentially operate on single-period or loosely coupled multi-period optimizations. This mismatch can produce suboptimal dispatch or uncompensated financial risk for the storage operator, who must manage a physical constraint that the market-clearing engine may not fully internalize or compensate.

The settlement structure compounds this tension. Traditional two-settlement systems (day-ahead and real-time) provide a basic framework for price formation. Nonetheless, they were devised for conventional power plants without considering the peculiarities of SOC-limited assets. 
%
%
%
%
Essentially, for a conventional generator, the capacity to deliver a day-ahead (DA) awarded schedule in real time (RT) is bounded basically only by its ramping constraints and installed nameplate rating--a quantity that is identical in both settlements. For a storage resource, the deliverable capacity at the RT stage instead depends on the SOC available at that moment, which can differ significantly from what was assumed in the DA optimization. {\color{black}
The DA market is cleared at once over the full horizon, relative to which battery duration is comparatively short, whereas the RT market is cleared sequentially over much shorter look-ahead windows, relative to which battery duration is comparatively long, and is conditioned on forecast information that is itself revised at every dispatch interval.} Every charge or discharge decision a storage resource makes, therefore, reshapes what it is physically capable of delivering in every subsequent interval, so the dispatch problem it faces is never truly static and cannot be characterized as a sequence of independent, memoryless choices. 
The resulting sequence of decisions can diverge substantially from what would have been optimal under the snapshot that a one-shot DA-style lookahead implicitly assumes \cite{Ghadimi2024_StochasticSearchStorageRollingForecasts}. Since an usual generator's marginal cost and capacity are in general memoryless---invariant to whatever forecast informed the RT clearing---this source of divergence has no analogue for conventional resources.

Beyond the general mismatch between one-shot DA scheduling and rolling-horizon RT clearing, ancillary and flexibility products actively compete for the same finite SOC, each pulling the storage resource's dispatch trajectory away from its DA-scheduled path for reasons that standard uplift payment schemes were never designed to price. 
Regulation awards are an example of this effect. For instance, in the CAISO market, since a storage resource must remain dispatchable to sustain a procured ancillary-service level for at least 30 minutes in real time, the market software enforces an ancillary-service SOC constraint that holds a fixed amount of SOC headroom or footroom in reserve. In this context, for example, consecutive regulation-down awards can charge a resource toward its upper limit, at which point the constraint forces an uneconomic discharge simply to preserve regulation-down headroom for the following interval~\cite{CAISO_2022_Energy_storage_enhancements}. A storage resource can therefore arrive at the RT interval with a different SOC than assumed, purely as a consequence of ramp-product obligations.

These structural and product-level issues represent, in practice, a persistent regulatory problem in California. The bid cost recovery (BCR) process was designed to provide uplift payments to conventional generators when revenues from energy and ancillary services did not cover the costs of following ISO dispatch instructions, guaranteeing cost recovery in exchange for bidding and dispatch compliance~\cite{CAISO_2024_Storage_bid_cost_recovery_enhancements}. With the arrival of grid-scale storage, this same framework has been largely extended to the new resource class without adapting it to SOC-dependent deliverability. CAISO's Department of Market Monitoring (DMM) has since documented that the resulting rules incentivize inefficient bidding. Since storage resources remain eligible for BCR payments even when their SOC cannot support their schedule---unlike conventional resources, which cannot receive BCR without fuel---batteries are insulated from losses tied to an inadequate SOC, reducing their exposure to real-time prices and encouraging strategies that maximize uplift payments ~\cite{CAISO_BCR_2025}. 

In this context, this paper formalizes the distinction between SOC-unlimited and
SOC-limited resources through a profit analysis grounded in LP duality,
and proposes a compensation rule that alleviates the SOC-induced profit
shortfall. Specifically, the contributions are:

\begin{enumerate}
    
    \item A formal analysis showing the intrinsic differences between the combined DA+RT profit of SOC-limited (storage) and SOC-unlimited resources (generators). The analysis demonstrates that, while SOC-unlimited resources have structural properties to guarantee cost recovery and non-negative total profit (assuming no fixed costs to make a direct comparison with storage systems), the resulting profit of SOC-limited resources can be driven negative when the DA and RT SOC levels diverge, with the {\color{black}sensitivity related to the SOC constraint's dual variable}.
    
    
    \item A \emph{max-rule} settlement in which the total storage profit is a maximum between (i) the standard day ahead plus incremental real-time compensation and (ii) only the real-time compensation. 
    The rule selects the less
    adverse of the two settlement outcomes reducing the
    make-whole BCR payments that the ISO must issue to storage
    operators without requiring a
    fundamental redesign of the two-settlement structure.
\end{enumerate}

{\color{black}The remainder of the paper is organized as follows. Section~\ref{sec:soc_drivers} discusses the catalysts for SOC divergence between DA and RT markers. Section~\ref{sec:profit_analysis} provides a formal profit comparison between SOC-unlimited and SOC-limited resources, deriving the SOC-gap threshold at which the combined profit turns negative for storage systems. Section~\ref{sec:proposed_rule} proposes an alternative settlement rule for SOC-limited resources. Section~\ref{sec:conceptual_examples} presents conceptual  examples comparing standard and proposed settlement rules. Section~\ref{sec:case_study} illustrates the performance of the proposed rule with a case study based on a realistic version of the California system. Section~\ref{sec:conclusion} concludes.

}

\section{Main catalysts for SOC divergence between DA and RT markets}
\label{sec:soc_drivers}

This section enumerates the three main mechanisms
responsible for differences between the SOC available to a
storage resource at the time of the RT clearing,
$v^{RT,\dagger}_{k,0}$ and the SOC that was assumed
when the DA schedule was determined, $v^{DA,\dagger}_{k,0}$. The existence of these mechanisms, summarized in Table \ref{tab:soc_drivers}, imply that the SOC gap $\Delta v_{k,0} \triangleq v^{DA,\dagger}_{k,0} - v^{RT,\dagger}_{k,0}$ is not merely a consequence of poor bidding by the storage operator, but is in large part driven by forces outside the storage operator's control. 
\vspace{-0.1cm}
\subsection{Exceptional dispatch}
\label{sec:ed}

Exceptional dispatch refers to manual dispatch instructions issued by
system operators \emph{outside} of the market optimization process to
maintain grid reliability or address transmission constraints that
cannot be resolved through normal market operations~\cite{CAISO_2009_exceptional_dispatch_white_paper}.
Since exceptional dispatches are not determined by the market
software, they do not set locational marginal prices and are not
reflected in the DA optimization.
When a storage resource receives an exceptional dispatch instruction
to charge (discharge) in an interval that was not planned at the DA
stage, the resulting SOC at the start of the next RT interval will
be higher (lower) than $v^{DA,\dagger}_{k,0}$, directly inducing a
non-zero gap $\Delta v_{k,0}$. In the CAISO market, this type of instruction warrants some form of make-whole compensation.

\vspace{-0.1cm}
\subsection{Local market power mitigation}
\label{sec:mitigation}

Local market power mitigation is the process employed by ISOs
to revise a resource's bids to a default energy bid level
when the resource is deemed to have market power at its
location~\cite{FERC_2024_local_power_market}.
For a storage resource, mitigation of its \emph{charging} bids can
prevent it from acquiring the SOC it would otherwise have secured
through economic bidding~\cite{CAISO_MarketSurveillanceCommittee_2024_local_power_market}. 
If the resource's willingness-to-pay for
charging is mitigated downward, it may be dispatched to charge less
(or not at all) in the hours preceding a high-value discharge window,
leaving $v^{RT,\dagger}_{k,0} < v^{DA,\dagger}_{k,0}$.
Conversely, mitigation of \emph{discharging} bids can force the
resource to discharge at a price below its actual opportunity cost,
depleting SOC prematurely.
CAISO analysis shows that the frequency of charging-side mitigation
events is less than 1\% for most resources. Nonetheless, charging-side mitigation events are more frequent during
the daytime solar hours immediately preceding the evening
net-peak, precisely the window in which
adequate SOC is most critical for system reliability~\cite{CAISO_BCR_2025}.

\subsection{Multi-interval optimization}
\label{sec:mio}

The real-time dispatch in most U.S. wholesale markets, including
CAISO's five-minute market, is implemented as a multi-interval
optimization~(MIO) with a look-ahead horizon 
(one binding interval plus subsequent advisory
intervals)~\cite{Hua_2019_Pricing_Multi-Interval_Real-Time}. {\color{black}
Notably, whereas the long horizon of the DA market makes most storage systems relatively short-duration, the much shorter window of the RT market makes the same systems relatively long-duration.}
The MIO can reposition a storage resource's output in the binding
interval---issuing a dispatch instruction that appears uneconomic
relative to the current interval's price---in order to preserve SOC
feasibility in advisory intervals where a higher-value dispatch is
anticipated.
Essentially, the MIO may schedule charging (discharging) in the current
binding interval even when the current RT price is above (below) the
resource's bid, because the advisory intervals signal that future
conditions will make such positioning optimal.
When those anticipated future conditions fail to materialize, the
storage resource has been dispatched out-of-merit in the binding
interval without the expected offsetting benefit in subsequent
intervals \cite{CAISO_MarketSurveillanceCommittee_2024_local_power_market}, resulting in a realized SOC that differs from both the DA
assumption and the MIO's projection.

\begin{table}[!t]
    \renewcommand{\arraystretch}{1.3}
    \caption{Catalysts of SOC Divergence Between DA and RT Clearings}
    \label{tab:soc_drivers}
    \centering
    \begin{tabular}{p{2.0cm} p{2.2cm} p{1.5cm} p{1.5cm}}
        \toprule
        \textbf{Driver} &
        \textbf{Mechanism} &
        \textbf{Attribution} &
        \textbf{SOC Gap Direction} \\
        \midrule
        Exceptional dispatch &
        ISO manual instruction outside market &
        ISO action &
        Either \\
        \addlinespace
        Market power mitigation &
        Bid revision to default energy bid &
        ISO/market design &
        Typically $v^{RT} < v^{DA}$ \\
        \addlinespace
        Multi-interval optimization &
        Out-of-merit binding-interval dispatch for advisory feasibility &
        Mixed &
        Either \\
        \bottomrule
    \end{tabular}
\end{table}
\vspace{-0.3cm}
\section{Profit analysis for SOC-unlimited and SOC-limited resources}
\label{sec:profit_analysis}

In this section, we analyze profits of conventional generators (SOC-unlimited) and storage systems (SOC-limited) in two-settlement electricity markets. In our analysis, we formalize the reason why the SOC-limited nature of storage systems make them more vulnerable to negative profits. 

\subsection{Profit analysis for SOC-unlimited resources}

Given a day-ahead price $\lambda^{DA}_{n}$ at the bus $n$ where the conventional generator $g$ is located, the generator $g$ will optimize its generation output $p^{G,DA}_g$ to maximize its profit, which is equivalent to minimize its negative profit as written below. An important assumption here is that the power plant $g$ does not have fixed costs, which are not considered in this analysis to focus the discussion on the SOC-unlimited versus SOC-limited comparison. 
\vspace{-0.1cm}
\begin{align}
    & q^{G,DA}_g (\lambda^{DA}_{n|g \in G_n}) = \underset{{\substack{ p^{G,DA}_g }}} {\text{Minimize}}  \hspace{0.2cm} - p^{G,DA}_g \Bigl ( \lambda^{DA}_{n|g \in G_n} - C_g \Bigr ) \label{gen_profit_DA_model_1} \\
    & \text{Subject to:} \notag \\
    & 0 \leq p^{G,DA}_g \leq \overline{P}^G_g: (\underline{\pi}^{G,DA}_g, \overline{\pi}^{G,DA}_g) \label{gen_profit_DA_model_2}, 
\end{align}

\noindent where \eqref{gen_profit_DA_model_1} maximizes the generator $g$ profit (compensation minus generation cost) and \eqref{gen_profit_DA_model_2} imposes generation limits.

The Lagrangian function related to \eqref{gen_profit_DA_model_1}--\eqref{gen_profit_DA_model_2} can be written as:
\vspace{-0.2cm}
\begin{align}
    & {\cal L}^G_g(p^{G,DA}_g, \underline{\pi}^{G,DA}_g, \overline{\pi}^{G,DA}_g) = - p^{G,DA}_g \Bigl ( \lambda^{DA}_{n|g \in G_n} - C_g \Bigr ) \notag\\
    & \hspace{65pt} - \underline{\pi}^{G,DA}_g p^{G,DA}_g + \overline{\pi}^{G,DA}_g \Bigr (p^{G,DA}_g - \overline{P}^G_g \Bigl)  
\end{align}

The first step to write the dual problem for \eqref{gen_profit_DA_model_1}--\eqref{gen_profit_DA_model_2} is writing the following expression. 
\vspace{-0.2cm}
\begin{align}
    & \underset{{\substack{ \underline{\pi}^{G,DA}_g, \overline{\pi}^{G,DA}_g }}} {\text{Maximize}}  \hspace{0.2cm} \Biggl \{ \underset{{\substack{ p^{G,DA}_g }}} {\text{Minimize}}  \hspace{0.2cm} {\cal L}^G_g(p^{G,DA}_g, \underline{\pi}^{G,DA}_g, \overline{\pi}^{G,DA}_g) \Biggl \} \notag \\
    & = \underset{{\substack{ \underline{\pi}^{G,DA}_g, \overline{\pi}^{G,DA}_g }}} {\text{Maximize}}  \hspace{0.2cm} \Biggl \{ \underset{{\substack{ p^{G,DA}_g }}} {\text{Minimize}}  \hspace{0.2cm} - p^{G,DA}_g \Bigl ( \lambda^{DA}_{n|g \in G_n} - C_g \Bigr ) \notag\\
    & \hspace{58pt} - \underline{\pi}^{G,DA}_g p^{G,DA}_g + \overline{\pi}^{G,DA}_g \Bigr( p^{G,DA}_g - \overline{P}^G_g \Bigl) \Biggl \} 
\end{align}

The dual objective function corresponding to \eqref{gen_profit_DA_model_1}--\eqref{gen_profit_DA_model_2} is:
\vspace{-0.2cm}
\begin{align}
    & \underset{{\substack{ \underline{\pi}^{G,DA}_g, \overline{\pi}^{G,DA}_g }}} {\text{Maximize}}  \hspace{0.2cm} - \overline{\pi}^{G,DA}_g \overline{P}^G_g  
\end{align}

Considering strong duality, primal and dual objective functions should return the same value at the optimal solution, so we have the following equality.
\vspace{-0.2cm}
\begin{align}
    & - p^{G,DA}_g \Bigl ( \lambda^{DA}_{n|g \in G_n} - C_g \Bigr ) = - \overline{\pi}^{G,DA}_g \overline{P}^G_g  
\end{align}

Since the decision variable $p^{G,DA}_g$, the dual variable $\overline{\pi}^{G,DA}_g$, and the parameter $\overline{P}^G_g$ are all positive, the result of $\Bigl ( \lambda^{DA}_{n|g \in G_n} - C_g \Bigr )$ will also be positive, consequently we have:
\vspace{-0.2cm}
\begin{align}
    & \lambda^{DA}_{n|g \in G_n} \geq C_g    
\end{align}

\noindent which ensures that the price $\lambda^{DA}_{n|g \in G_n}$ will be higher than the cost $C_g$ and, therefore, cost recovery is guaranteed in the day-ahead market when fixed costs are not considered. 
%

Given a certain day-ahead price $\lambda^{RT}_{n}$ at the bus $n$ where the generator $g$ is located, the generator $g$ will optimize its generation output $p^{G,RT}_g$ to maximize its profit, which is equivalent to minimize its negative profit as written below.  
\vspace{-0.2cm}
\begin{align}
    & q^{G,RT}_g (\lambda^{RT}_{n|g \in G_n}) =\underset{{\substack{ p^{G,RT}_g }}} {\text{Minimize}}  \hspace{0.2cm} - p^{G,RT}_g \Bigl ( \lambda^{RT}_{n|g \in G_n} - C_g \Bigr )  \label{gen_profit_RT_model_1} \\
    & \text{Subject to:} \notag \\
    & 0 \leq p^{G,RT}_g \leq \overline{P}^G_g: (\underline{\pi}^{G,RT}_g, \overline{\pi}^{G,RT}_g) \label{gen_profit_RT_model_2} 
\end{align}

Model \eqref{gen_profit_RT_model_1}--\eqref{gen_profit_RT_model_2} has the same structure of \eqref{gen_profit_DA_model_1}--\eqref{gen_profit_DA_model_2} and the main difference is the real-time price $\lambda^{RT}_{n|g \in G_n}$ which can be significantly different to the day-ahead price $\lambda^{DA}_{n|g \in G_n}$.

In this case, considering only the RT market, following the same rationale previously described for the DA market, we have:

\begin{align}
    & \lambda^{RT}_{n|g \in G_n} \geq C_g    
\end{align}

\noindent which means cost recovery is also guaranteed in the real-time market for generator $g$. 
%

The total profit for generator $g$ considering the outcomes of both day-ahead and real-time market can be computed as:

\begin{align}
    & TP^G_g = p^{G,DA}_g \Bigl ( \lambda^{DA}_{n|g \in G_n} - C_g \Bigr ) \notag \\
    & \hspace{75pt} + \Bigl ( p^{G,RT}_g - p^{G,DA}_g \Bigr ) \Bigl ( \lambda^{RT}_{n|g \in G_n} - C_g \Bigr )
\end{align}

The value of $TP^G_g$ may be negative only if $p^{G,RT}_g < p^{G,DA}_g$ and $\lambda^{RT}_{n|g \in G_n} > \lambda^{DA}_{n|g \in G_n}$. Nonetheless, the capacity of generator $g$ does not prevent its generation from being the same or higher in real time compared to day ahead. Furthermore, if, in fact, the real time price is higher than the day ahead price, the power plant will have an increased motivation to provide more generation in real time.    

\subsection{Profit analysis for SOC-limited resources} \label{subsec:storage}

Given an available state-of-charge $v^{DA,\dagger}_{k,0}$ and a day-ahead price $\lambda^{DA}_{n}$ at the bus $n$ where the storage system $k$ is located, the storage $k$ will optimize its charging and discharging output $p^{ch,DA}_k$ and $p^{dis,DA}_k$ to maximize its profit, which is equivalent to minimize its negative profit as written below.

\begin{align}
    & q^{St,DA}_k  (\lambda^{DA}_{n|k \in K_n}, v^{DA,\dagger}_{k,0}) = \notag\\
    & \hspace{10pt} \underset{{\substack{ p^{dis,DA}_k, p^{ch,DA}_k, v^{DA}_{k} }}} {\text{Minimize}}  \hspace{0.2cm} - \Biggl [ \Bigl ( p^{dis,DA}_k - p^{ch,DA}_k \Bigr ) \lambda^{DA}_{n|k \in K_n} \notag\\
    &\hspace{80pt} - \Bigl ( p^{dis,DA}_k + p^{ch,DA}_k \Bigr ) C^{degradation}_k \Biggr ] \label{storage_profit_DA_model_1} \\
    & \text{Subject to:} \notag \\
    & v^{DA}_{k} - v^{DA,\dagger}_{k,0} = - \frac{p^{dis,DA}_k}{\eta_k} + p^{ch,DA}_k \eta_k: ({\pi}^{St,DA,1}_k) \label{storage_profit_DA_model_2} \\
    & 0 \leq p^{ch,DA}_k \leq \overline{P}^{ch}_k: ({\pi}^{St,DA,2}_k, {\pi}^{St,DA,3}_k) \label{storage_profit_DA_model_3} \\
    & 0 \leq p^{dis,DA}_k \leq \overline{P}^{dis}_k: ({\pi}^{St,DA,4}_k, {\pi}^{St,DA,5}_k) \label{storage_profit_DA_model_4} \\
    & 0 \leq v^{DA}_k \leq \overline{V}_k: ({\pi}^{St,DA,6}_k, {\pi}^{St,DA,7}_k) \label{storage_profit_DA_model_5} 
\end{align}

Following the same rationale developed for generators and relying on strong duality, we can write as follows that primal and dual objective functions should return the same value at the optimal solution.

\begin{align}
    & \Bigl ( p^{dis,DA}_k - p^{ch,DA}_k \Bigr ) \lambda^{DA}_{n|k \in K_n} - \Bigl ( p^{dis,DA}_k \notag\\ 
    & \hspace{10pt} + p^{ch,DA}_k \Bigr ) C^{degradation}_k = {\pi}^{St,DA,1}_k v^{DA,\dagger}_{k,0} \notag\\ 
    & \hspace{10pt} + {\pi}^{St,DA,3}_k \overline{P}^{ch}_{k} + {\pi}^{St,DA,5}_k \overline{P}^{dis}_{k} + {\pi}^{St,DA,7}_k \overline{V}_{k} \label{strong_duality_storage_DA}  
\end{align}

In \eqref{strong_duality_storage_DA}, the only term in the right-hand side that can assume a negative value is the first one, ${\pi}^{St,DA,1}_k v^{DA,\dagger}_{k,0}$, since ${\pi}^{St,DA,1}_k$ is the only dual variable in \eqref{storage_profit_DA_model_1}--\eqref{storage_profit_DA_model_5} that is free to be either positive or negative, whereas the other dual variables are restricted to be positive. In this context, storage $k$ would have a resulting negative profit if ${\pi}^{St,DA,1}_k$ is negative enough and, in fact, the impact of a change in the state-of-charge $v^{DA,\dagger}_{k,0}$ on the profit of storage $k$ is ${\pi}^{St,DA,1}_k$, i.e., $\pd{{q^{St,DA}_k  (\lambda^{DA}_{n|g \in G_n}, v^{DA,\dagger}_{k,0})}}{{v^{DA,\dagger}_{k,0}}} = {\pi}^{St,DA,1}_k$. The optimization of the storage system $k$'s profit in the DA market will then closely consider both $v^{DA,\dagger}_{k,0}$ and ${\pi}^{St,DA,1}_k$ to seek a positive profit in the DA market.

Analogously, given real-time price and state-of-charge availability, we can write a profit-maximization optimization problem for storage $k$ in real time (as in \eqref{storage_profit_DA_model_1}--\eqref{storage_profit_DA_model_5}) and the equality between primal and dual objective functions is:  

\begin{align}
    & \Bigl ( p^{dis,RT}_k - p^{ch,RT}_k \Bigr ) \lambda^{DA}_{n|k \in K_n} - \Bigl ( p^{dis,RT}_k \notag\\ 
    & \hspace{10pt} + p^{ch,RT}_k \Bigr ) C^{degradation}_k = {\pi}^{St,RT,1}_k v^{RT,\dagger}_{k,0} \notag\\
    & \hspace{10pt} + {\pi}^{St,RT,3}_k \overline{P}^{ch}_{k} + {\pi}^{St,RT,5}_k \overline{P}^{dis}_{k} + {\pi}^{St,RT,7}_k \overline{V}_{k} \label{strong_duality_storage_RT}  
\end{align}

Similarly to the DA market case, according to \eqref{strong_duality_storage_RT}, in real time, the profit of storage $k$ will depend on the values assumed by the dual variable ${\pi}^{St,RT,1}_k$ (which can be either positive or negative) and the available state-of-charge $v^{RT,\dagger}_{k,0}$. The optimization of the storage system $k$'s profit in the RT market, again similarly to the DA market, will pay close attention to ${\pi}^{St,RT,1}_k$ and $v^{RT,\dagger}_{k,0}$ to aim for a positive profit in real time.

In this context, in each of these two markets (day ahead and real time) individually, the storage system $k$ can take its most appropriate decision to maximize its profit according to the current circumstances, specially the available state-of-charge. However, the total profit of storage systems usually considers both DA and RT components as follows.

\begin{align}
    & TP^{St}_k = \Biggl [ p^{dis,DA}_k - p^{ch,DA}_k \Biggr ] \lambda^{DA}_{n|k \in K_n} \notag\\
    & \hspace{5pt} - \Biggl [ p^{dis,DA}_k + p^{ch,DA}_k \Biggr ] C^{degradation}_k + \Biggl [  \Bigl ( p^{dis,RT}_k  \notag \\
    & \hspace{5pt} - p^{dis,DA}_k \Bigr) - \Bigl ( p^{ch,RT}_k - p^{ch,DA}_k \Bigr) \Biggr ] \lambda^{RT}_{n|k \in K_n} - \Biggl [ \Bigl ( p^{dis,RT}_k \notag \\
    & \hspace{5pt}  - p^{dis,DA}_k \Bigr) + \Bigl ( p^{ch,RT}_k - p^{ch,DA}_k \Bigr) \Biggr ] C^{degradation}_k\label{combined_profit_storage}  
\end{align}

Hence, if $v^{DA,\dagger}_{k,0} \neq v^{RT,\dagger}_{k,0}$, we may forcibly have $p^{dis,RT}_k < p^{dis,DA}_k$ and/or $p^{ch,RT}_k > p^{ch,DA}_k$ when $\lambda^{RT}_{n|k \in K_n} > \lambda^{DA}_{n|k \in K_n}$, which can drive the total profit, $TP^{St}_k$, of storage system $k$ to a negative value.
\vspace{0.1cm}
\begin{prop}[Attainability of negative combined storage profit]
\label{prop:negative_profit}

Let $\Pi^{DA}_k$ 
and $\Pi^{RT}_k$ 
denote the optimal DA-only and
RT-only profits of storage unit $k$, evaluated at their
respective available SOC levels.
Define the SOC gap
\begin{align}
    \Delta v_{k,0}
    \;\triangleq\;
    v^{DA,\dagger}_{k,0} - v^{RT,\dagger}_{k,0}.
    \label{eq:soc_gap_def}
\end{align}

Under the standard two-settlement rule the combined DA$+$RT
profit is

\begin{align}
    &TP^{St}_k
    = \Pi^{DA}_k
    \;+\;
    \Bigl[\bigl(p^{dis,RT}_k - p^{dis,DA}_k\bigr)
          -\bigl(p^{ch,RT}_k  \notag\\
    & \hspace{30pt} - p^{ch,DA}_k \bigr)\Bigr]
    \lambda^{RT}_{n}
    -\Bigl[\bigl(p^{dis,RT}_k - p^{dis,DA}_k\bigr)\notag\\
    & \hspace{30pt} + \bigl(p^{ch,RT}_k  - p^{ch,DA}_k \bigr)\Bigr]
    C^{degradation}_k.
    \label{eq:TP_storage}
\end{align}

Then $TP^{St}_k < 0$ is attainable.  Specifically, for a high price period when ${p}^{dis,DA}_k = \eta_k v^{DA,\dagger}_{k,0}$, if $\Delta v_{k,0} > 0$
, the combined profit satisfies

\begin{align}
    TP^{St}_k
    \;\leq\;
    \Pi^{DA}_k
    \;-\;
    \pi^{St,RT,1}_k \,\Delta v_{k,0},
    \label{eq:TP_upper_bound}
\end{align}

\noindent where $\pi^{St,RT,1}_k$ 
is the dual variable
associated with the SOC balance constraint (the real-time
analogue of \eqref{storage_profit_DA_model_2}) at the RT optimum.
For sufficiently large $\Delta v_{k,0}$, the right-hand side
of \eqref{eq:TP_upper_bound} becomes negative, so
$TP^{St}_k < 0$.
\end{prop}

\begin{proof}
We proceed in three steps.

\textit{Step 1 (RT dispatch feasibility).}
When $\Delta v_{k,0} > 0$, the RT-realized SOC
$v^{RT,\dagger}_{k,0} = v^{DA,\dagger}_{k,0} - \Delta v_{k,0}$
is strictly smaller than the SOC assumed in the DA optimization. The maximum discharging available in real time, $\overline{p}^{dis,RT}_k$, is lower than the maximum discharging available in the day-ahead optimization, $\overline{p}^{dis,DA}_k$, if $\Delta v_{k,0}>0$. More specifically, we have: 
\vspace{-0.2cm}
\begin{align}
    & p^{dis,RT}_k
    \leq
    \overline{p}^{dis,RT}_k = \eta_k\,v^{RT,\dagger}_{k,0} =
    \eta_k\bigl(v^{DA,\dagger}_{k,0} - \Delta v_{k,0}\bigr) \notag \\
    & = \overline{p}^{dis,DA}_k - \eta_k\,\Delta v_{k,0}.
    \label{eq:discharge_constraint}
\end{align}

Hence $p^{dis,RT}_k < \overline{p}^{dis,DA}_k$ whenever $\Delta v_{k,0}>0$. For a high price period when ${p}^{dis,DA}_k = \overline{p}^{dis,DA} = \eta_k v^{DA,\dagger}_{k,0}$ , we then have $p^{dis,RT}_k < {p}^{dis,DA}_k$.

\textit{Step 2 (RT deviation cost).}
The two-settlement RT component of $TP^{St}_k$ captures the
deviation from the DA schedule:
\begin{align}
    TP^{St}_k - &\Pi^{DA}_k
    = \bigl(p^{dis,RT}_k - p^{dis,DA}_k\bigr)
       \bigl(\lambda^{RT}_{n} - C^{degradation}_k\bigr)
    \notag\\
    &\quad 
       -\bigl(p^{ch,RT}_k - p^{ch,DA}_k\bigr)
       \bigl(\lambda^{RT}_{n} + C^{degradation}_k\bigr).
    \label{eq:RT_component}
\end{align}

Considering $p^{dis,RT}_k < p^{dis,DA}_k$ (from Step~1) and
$\lambda^{RT}_{n} > C^{degradation}_k$ 
, the first term
is strictly negative. Substituting the bound
\eqref{eq:discharge_constraint} (again assuming ${p}^{dis,DA}_k = \overline{p}^{dis,DA}_k$) into \eqref{eq:RT_component}
and dropping the non-positive charging deviation term yields
\vspace{-0.2cm}
\begin{align}
    TP^{St}_k - \Pi^{DA}_k
    \;\leq\;
    -\eta_k\,\Delta v_{k,0}
    \bigl(\lambda^{RT}_{n} - C^{degradation}_k\bigr).
    \label{eq:rt_ineq}
\end{align}

\textit{Step 3 (Connection to the SOC dual variable).}
Based on LP sensitivity theory applied to the RT problem (the
real-time analogue of
\eqref{storage_profit_DA_model_1}--\eqref{storage_profit_DA_model_5}),
the marginal value of relaxing the RT-available SOC is
\vspace{-0.2cm}
\begin{align}
    \pi^{St,RT,1}_k
    =
    \frac{\partial\,q^{St,RT}_k}{\partial\,v^{RT,\dagger}_{k,0}}.
    \label{eq:dual_sensitivity}
\end{align}

When $0 < p^{dis,RT}_k < \overline{P}^{dis}_k$ the analogous of constraints \eqref{storage_profit_DA_model_4} are not binding for the RT problem and, therefore, we have 
$\pi^{St,RT,1}_k = \eta_k\bigl(\lambda^{RT}_{n} -
C^{degradation}_k\bigr)$ 
meaning that the shadow price of an additional
MWh of RT-available SOC equals the additional discharge revenue
net of degradation cost, scaled by efficiency. Substituting into
\eqref{eq:rt_ineq}:
\vspace{-0.2cm}
\begin{align}
    TP^{St}_k
    \;\leq\;
    \Pi^{DA}_k
    \;-\;
    \pi^{St,RT,1}_k\,\Delta v_{k,0}.
    \label{eq:TP_bound_final}
\end{align}

Since $\Pi^{DA}_k$ is finite and for $\pi^{St,RT,1}_k > 0$, the
right-hand side of \eqref{eq:TP_bound_final} is a 
decreasing function of $\Delta v_{k,0}$.  Thus, choosing
\vspace{-0.2cm}
\begin{align}
    \Delta v_{k,0}
    \;>\;
    \frac{\Pi^{DA}_k}{\pi^{St,RT,1}_k}
    \label{eq:threshold}
\end{align}
makes the right-hand side negative, so $TP^{St}_k < 0$.
Such a gap is attainable because the three drivers enumerated in
Section~\ref{sec:soc_drivers} can produce $\Delta v_{k,0}$ of
arbitrary magnitude up to $v^{DA,\dagger}_{k,0}$.
\end{proof}

\noindent
\textbf{Remark.}
Equation~\eqref{eq:TP_bound_final} makes the asymmetry with
generators precise. For a conventional generator,
$\partial\,q^{G,DA}_g / \partial\,\overline{P}^G_g \geq 0$
and capacity is identical
across settlements, so the analogue of $\Delta v_{k,0}$ is
identically zero.  For storage, $\pi^{St,RT,1}_k$ plays the
role of the ``value of SOC'' at the real-time margin, and any
positive gap $\Delta v_{k,0}$ linearly erodes the combined
profit at rate $\pi^{St,RT,1}_k$. The storage participant can
achieve a positive DA-only profit $\Pi^{DA}_k > 0$ by choosing
its DA dispatch optimally, yet still realize $TP^{St}_k < 0$
whenever $\Delta v_{k,0}$ exceeds the threshold
\eqref{eq:threshold}, which is a situation that no analogous manipulation of generator capacity can produce.
\vspace{-0.2cm}
\section{Proposed loss-reduction-oriented max rule}
\label{sec:proposed_rule}

Let $\Pi^{DA}_k + \Pi^{RT,incr}_k$ denote the standard
two-settlement profit of storage $k$, where
$\Pi^{RT,incr}_k$ captures the incremental real-time
deviations from the DA schedule. Also, let
$\Pi^{RT\text{-}only}_k(v^{RT,\dagger}_{k,0})$ denote
the optimal RT-only profit evaluated from the realized
SOC $v^{RT,\dagger}_{k,0}$. We define the proposed max-rule total profit
as:

\begin{align}
    & TP^{St,\text{new}}_k = \max \Bigl(\Pi^{DA}_k + \Pi^{RT,incr}_k,\Pi^{RT\text{-}only}_k\Bigr) \label{eq:max_rule_formal} 
\end{align}

In the next subsections, we discuss (i) the loss-reduction property of the proposed max rule, (ii) the possibility of gaming behavior, and (iii) the complementarity of the max rule with the CAISO's options under consideration to alleviate storages' shortfalls. 
\vspace{-0.2cm}
\subsection{Loss-reduction property}

The loss reduction property of the proposed max rule is formalized via the following proposition.

\begin{prop}[The max rule's loss reduction property]
\label{prop:max_rule}
    The max rule defined in \eqref{eq:max_rule_formal} weakly dominates the standard two-settlement profit calculation  in every scenario. Furthermore, when $\Pi^{RT\text{-}only}_k \geq 0$ and $\Pi^{DA}_k + \Pi^{RT,incr}_k < 0$, the max rule strictly eliminates the loss that could trigger a BCR event.
\end{prop}


Proposition~\ref{prop:max_rule} establishes that the max rule weakly dominates the standard two-settlement profit outcome in every scenario. Since $\max(a, b) \geq a$ for any $a, b \in
\mathbb{R}$, the storage participant is never worse off under the max rule than under the standard two-settlement profit rule alone. The key insight is that, by evaluating each market settlement
against the SOC that is \emph{actually available} at that clearing, the RT-only leg anchors the settlement
to a profit level that reflects the storage unit's true physical state. When the RT-only profit is non-negative and the standard two-settlement profit is negative (e.g, the configuration identified in Proposition~\ref{prop:negative_profit} as arising when $\Delta v_{k,0}$ exceeds the threshold \eqref{eq:threshold}) the max rule strictly eliminates the loss. When both legs are negative, the max rule
selects the less adverse outcome, reducing but not eliminating the loss.

It is, nonetheless, important to note what the max rule does \emph{not} do. As previously mentioned, it does not guarantee non-negative profit in all circumstances. If both the standard two-settlement profit and
the RT-only profit are negative, the storage participant still incurs a loss, albeit a smaller one than under the standard two-settlement profit alone. In addition, the max rule does not guarantee the same profit
as would have been earned had the SOC gap never occurred. The RT-only profit
$\Pi^{RT\text{-}only}_k(v^{RT,\dagger}_{k,0})$ can be lower than $\Pi^{DA}_k$ when the SOC gap
reduces the energy available for high-value RT discharge windows. The rule reduces or eliminates the
\emph{shortfall} but does not eliminate the \emph{opportunity cost} of the SOC reduction. Essentially, the opportunity cost is a consequence of a physical state that the market cannot undo, whereas the shortfall is a consequence of a settlement rule that can be adapted to better accommodate SOC-limited resources.
\vspace{-0.2cm}
\subsection{Can the Rule Be Gamed?}
\label{sec:discussion_gaming}

A natural concern about the max-rule \eqref{eq:max_rule_formal} is whether a strategic storage operator could exploit the $\max(\cdot,\cdot)$ operator to earn compensations above the
competitive benchmark. Since the rule selects the higher of two
market-determined quantities---the standard two-settlement profit $\Pi^{DA}_k + \Pi^{RT,incr}_k$ and the fresh
RT-only profit $\Pi^{RT\text{-}only}_k$ anchored to the
\emph{realized} SOC $v^{RT,\dagger}_{k,0}$---a gaming strategy
must either inflate $\Pi^{DA}_k + \Pi^{RT,incr}_k$ or inflate
$\Pi^{RT\text{-}only}_k$ beyond what competitive dispatch would
produce. We consider the following candidate strategies.

\subsubsection{Strategic SOC Depletion Before RT}

A storage operator might deliberately deplete its SOC between the
DA and RT clearings---for instance through self-scheduling
discharge---to engineer a low $v^{RT,\dagger}_{k,0}$ that shifts
the max in \eqref{eq:max_rule_formal} toward
$\Pi^{RT\text{-}only}_k$. This strategy is self-defeating.
Because $\Pi^{RT\text{-}only}_k$ is computed against the
\emph{realized} SOC, a lower $v^{RT,\dagger}_{k,0}$ directly
reduces the energy available for RT discharge, shrinking
$\Pi^{RT\text{-}only}_k$.
Moreover, the pre-RT discharge that depletes the SOC is itself
settled at RT prices, so the operator bears its cost immediately.
The only scenario in which deliberate depletion could be
beneficial is one in which the RT price in the high-value window
is substantially higher than anticipated at the DA stage. In
that case, however, the RT energy settlement already compensates
the operator through standard LMP-based revenues, independently
of the max rule.

\subsubsection{Strategic overbidding in DA} 

A storage operator might submit DA discharge bids and/or DA charge bids in order to maximize $\Pi^{DA}_k$,
reasoning that the term $\Pi^{DA}_k + \Pi^{RT,incr}_k$ can be somehow maximized without the storage having to actually contribute in the RT market. 
However, again this strategy would not be particularly more rewarding than the proposed rule is intended to be. Under our max rule, an undeliverable DA schedule in RT can produce a highly negative  $\Pi^{RT,incr}_k$ that can also drive $\Pi^{DA}_k + \Pi^{RT,incr}_k$ to be low and potentially negative and the storage would rely on its performance in the RT market as expected. Furthermore, the proposed rule would limit the BCR payments
that are currently delivered even when storage systems have SOCs that cannot support their DA schedules. The max rule replaces open-ended uplift payments with a capped,
market-determined quantity---the larger of the two legs---which
at least bounds the potential payment extraction. Finally, bid mitigation rules already present in CAISO's market
design constrain the most egregious forms of strategic DA
overbidding, and applying mitigation
to the DA bids that enter the max-rule computation imposes a
further check on this channel.


\vspace{-0.2cm}
\subsection{Relationship to CAISO's options under consideration to alleviate storages' shortfalls and BCR triggering}
\label{sec:discussion_caiso}

CAISO's Storage Design and Modeling working group has identified
three conceptual frameworks for redesigning real-time storage
uplift~\cite{CAISO_BCR_2025}. Next, we discuss each in relation to the proposed max rule.

\subsubsection{Option A: Mitigation-Focused Make-Whole}

Option~A, proposed by CAISO's Department of Market Monitoring (DMM), would redesign BCR rules to assume no eligibility for batteries by default and add eligibility only under specific circumstances where it is deemed appropriate, with a particular focus on local market power mitigation as a warranted driver of uplift. 
The make-whole calculation under Option~A is based on counterfactual settlement provisions developed for exceptional dispatches issued to hold state of charge, computing the difference between the resource's maximum potential RT revenues without and with the relevant market action.

The key disadvantage of Option~A is the implementation burden since it requires a computationally intensive counterfactual
calculation much more often than the current practice. On the other hand, the proposed rule can be directly integrated into the current practice to decrease the frequency of negative profits for storage systems (due to their SOC-limited nature) through the single transparent
formula \eqref{eq:max_rule_formal}, which needs no
counterfactual market re-runs.

\subsubsection{Option B: Bidding-Consistency Test}

Option~B 
would remove negative revenues (shortfalls) from the
BCR calculation when a forced buy-back or sell-back is caused
by the resource's own actions, while retaining make-whole
eligibility for shortfalls caused by ISO market
design. 
The attribution mechanism
relies on a \emph{bidding-consistency test}. A storage
resource is deemed to have bid ``consistently'' with RT
conditions if its discharge (charge) bid is at or above
(below) the RT default energy bid~(DEB), and shortfalls
arising from ``inconsistent'' bids are excluded from the
BCR calculation.

The max rule \eqref{eq:max_rule_formal} shares Option~B's
core insight that bidder-driven shortfalls should not warrant
compensation while ISO-driven shortfalls should, but differs
fundamentally in mechanism.
Option~B requires a period-by-period comparison of submitted
bids to the RT DEB---a comparison that is both data-intensive
and sensitive to the accuracy of the DEB as a proxy for true
opportunity costs.
CAISO has noted that the current real-time DEB relies on
day-ahead market outputs and may over- or under-estimate
real-time opportunity costs~\cite{CAISO_BCR_2025}, which
directly undermines the reliability of Option~B's
consistency test.
The max rule \eqref{eq:max_rule_formal} avoids this
dependency entirely. It requires only the DA profit
(already computed by the DA market software) and the
RT-only profit (computed from telemetered SOC data and
RT prices), neither of which depends on DEB accuracy.
Moreover, the inequality in \eqref{eq:max_rule_formal}
provides an unconditional non-worsening guarantee relative
to the standard two-settlement profit outcome, whereas Option~B's coverage depends
on whether the DEB correctly classifies each shortfall as
bidder-driven or ISO-driven.

\subsubsection{Option C: Full Elimination of Uplift Through
               Structural Reform}

Option~C 
would eliminate the need for
any storage uplift mechanism by removing the market-design
constraints that drive SOC divergence in the first place,
including exempting storage from local market power
mitigation, deactivating the ancillary service SOC
constraint, and revising the upper and lower economic limit
calculations. 
In exchange, storage
resources would be subject to new penalty mechanisms for
failing to deliver 
awards, and
day-ahead and real-time BCR would be eliminated entirely.

Option~C is the most structurally ambitious of the three
frameworks and would, if implemented correctly, address
the root causes of SOC divergence enumerated in
Section~\ref{sec:soc_drivers} rather than its consequences.
However, these potential benefits come with significant
regulatory and feasibility risks. Eliminating market power
mitigation for storage exposes ratepayers to the risk of
high local energy prices and may create a worse discriminatory
treatment relative to other resources. 
Furthermore, simply eliminating the ancillary service SOC
constraint without implementing full real-time 
re-optimization would be insufficient to preserve
reliability.

The max rule \eqref{eq:max_rule_formal} is best understood
as a complement to, rather than a substitute for, Option~C's
structural agenda.
Since $TP^{St,\text{new}}_k \geq \Pi^{DA}_k +
\Pi^{RT,incr}_k$ by construction, the rule delivers
immediate improvement over the standard two-settlement rule 
for storage participants under the \emph{existing} market
structure, providing relief while the more fundamental
market-design reforms envisioned in Option~C can be carefully considered and debated. 

\vspace{-0cm}
\section{Conceptual examples}
\label{sec:conceptual_examples} 

In this section, we provide three conceptual examples of how the profit for a storage system would be computed under the standard two settlement rule and under the proposed max rule. The first and second examples illustrate, respectively, {\it buy back} and {\it sell back} situations, while the third example involves both {\it buy back} and {\it sell back} circumstances within the same time window.
\vspace{-0.3cm}
\subsection{Conceptual example 1: buy back situation}
\label{sec:buy_back_example__concetual_example_1}

Let a storage asset have a power capacity of
$\overline{P}^{ch}_k = \overline{P}^{dis}_k = 10$\,MW and an energy
capacity of $\overline{V}_k = 10$\,MWh, with an ideal round-trip efficiency
$\eta_k = 1$, with DA and RT prices reported in Table \ref{tab:prices_concetual_example_1}.
The initial state of charge is $v^{DA}_{k,0} = v^{RT}_{k,0} = 0$\,MWh at the start of
both the day-ahead and real-time market optimizations.
We consider 3 time periods in this example for both markets.

\begin{table}[!bt]
    \caption{Conceptual example 1 -- day-ahead and real-time prices (\$/MWh)}
    \label{tab:prices_concetual_example_1}
        \centering
        \renewcommand{\arraystretch}{1.2}
        \begin{tabular}{lccccccc}
        \toprule
        & \multicolumn{3}{c}{\textbf{Period}} \\
        \cmidrule(lr){2-7}
        \textbf{Price} & $t=1$ & $t=2$ & $t=3$ \\
        \midrule
        $\lambda^{DA}_t$ (\$/MWh) & {\color{white}0}$75$ & {\color{white}0}$75$ & $100$  \\
        $\lambda^{RT}_t$ (\$/MWh) & $100$ & $150$ & $500$ \\
        \bottomrule
    \end{tabular}
\end{table}

\begin{table}[!bt]
    \footnotesize
    \caption{Conceptual example 1 -- DA schedule, RT actual dispatch, and SOC trajectories}
    \label{tab:dispatch_concetual_example_1}
    \centering
    \renewcommand{\arraystretch}{1.1}
    \begin{tabular}{lcccc}
        \toprule
        \textbf{Quantity} & \textbf{Init.} & $t=1$ & $t=2$ & $t=3$ \\
        \midrule
        $p^{ch,DA}_{k,t}$ (MW)  & -- & $10$ & $0$ & $0$ \\
        $p^{dis,DA}_{k,t}$ (MW) & -- &  $0$ &  $0$ &  $10$ \\
        $v^{DA}_{k,t}$ (MWh)    & $0$ & $10$ & $10$ & $0$ \\
        \midrule
        $p^{ch,RT}_{k,t}$ (MW)  & -- & $10$ & $0$ &  $0$ \\
        $p^{dis,RT}_{k,t}$ (MW) & -- &  $0$ &  $10^{\dagger}$ & $0$  \\
        $v^{RT}_{k,t}$ (MWh)    & $0$ & $10$ & $0$ & $0$ \\
        \midrule
        $\Delta p^{ch}_{k,t}$ (MW)  & -- & $0$ & $0$ & $0$ \\
        $\Delta p^{dis}_{k,t}$ (MW) & -- & $0$ & $+10$ & $-10$ \\
        \bottomrule
        \multicolumn{5}{l}{$\dagger$~Forced discharge to comply with an exceptional dispatch. 
        }
    \end{tabular}
\end{table}

In this example, as described in Table \ref{tab:dispatch_concetual_example_1}, the storage resource is scheduled to charge in the first DA period and discharge in the third one. In the RT market however, a discharge instruction is imposed in the second period due to an exceptional dispatch for needed load supply 
and, as consequence, the storage is forced to buy back energy in the last RT period.  The day-ahead, real-time as well as total profits under the standard two settlement outcome and under the proposed rule are computed as follows ignoring the degradation costs without loss of generality.
\vspace{-0.2cm}
\begin{align}
  \Pi^{DA}_k
    &= \sum_{t \in T} \lambda^{DA}_{t} (p^{dis,DA}_{k,t} - p^{ch,DA}_{k,t}) = \$250
  \label{eq:ex_DA_profit_concetual_example_1} \\
  \Pi^{RT}_k
    &= \sum_{t \in T} \lambda^{RT}_{t} (p^{dis,RT}_{k,t} - p^{ch,RT}_{k,t}) = \$500 
  \label{eq:ex_RT_profit_profit_concetual_example_1} \\ 
  \Pi^{RT,incr}_k
    &= \sum_{t \in T}
         \lambda^{RT}_t
         \bigl(\Delta p^{dis}_{k,t} - \Delta p^{ch}_{k,t}\bigr) = -\$3500
  \label{eq:RT_incr_formula_profit_concetual_example_1} \\
  TP^{St}_k &= \Pi^{DA}_k + \Pi^{RT,incr}_k = -\$3250 \label{eq:total_profit_standard_2-settlemnet_concetual_example_1}\\
  TP^{St,\text{new}}_k &= \max \Bigl(\Pi^{DA}_k + \Pi^{RT,incr}_k,\Pi^{RT\text{-}only}_k\Bigr) = \$500 \label{eq:total_profit_max-rule_concetual_example_1}
\end{align}

In this situation, the storage resource incurs in a loss according to the standard settlement rule \eqref{eq:total_profit_standard_2-settlemnet_concetual_example_1} that might trigger a substantial BCR payment. Following the proposed rule \eqref{eq:total_profit_max-rule_concetual_example_1}, nevertheless, implies in a positive profit comparable to what was expected in the DA market. 


\subsection{Conceptual example 2: sell back situation}
\label{sec:sell_back_example_concetual_example_2}

Let a storage asset have a power capacity of
$\overline{P}^{ch}_k = \overline{P}^{dis}_k = 20$\,MW and an energy
capacity of $\overline{V}_k = 20$\,MWh, with an ideal round-trip efficiency
$\eta_k = 1$, with DA and RT prices reported in Table \ref{tab:prices_concetual_example_2}. The initial state of charge is $v^{DA}_{k,0} = v^{RT}_{k,0} = 0$\,MWh at the start of
both the day-ahead and real-time market optimizations.
We consider 4 time periods in this example for both markets.

\begin{table}[!bt]
    \caption{Conceptual example 2 -- day-ahead and real-time prices (\$/MWh)}
    \label{tab:prices_concetual_example_2}
        \centering
        \renewcommand{\arraystretch}{1.2}
        \begin{tabular}{lccccccc}
        \toprule
        & \multicolumn{4}{c}{\textbf{Period}} \\
        \cmidrule(lr){2-7}
        \textbf{Price} & $t=1$ & $t=2$ & $t=3$ & $t=4$ \\
        \midrule
        $\lambda^{DA}_t$ (\$/MWh) & {\color{white}0}$50$ & {\color{white}0}$70$ & {\color{white}0}$50$ & $150$  \\
        $\lambda^{RT}_t$ (\$/MWh) & $180$ & $320$ & $100$ & $400$ \\
        \bottomrule
    \end{tabular}
\end{table}

\begin{table}[!bt]
    \caption{Conceptual example 2 -- DA schedule, RT actual dispatch, and SOC trajectories}
    \label{tab:dispatch_concetual_example_2}
    \centering
    \renewcommand{\arraystretch}{1.1}
    \begin{tabular}{lccccc}
        \toprule
        \textbf{Quantity} & \textbf{Init.} & $t=1$ & $t=2$ & $t=3$ & $t=4$ \\
        \midrule
        $p^{ch,DA}_{k,t}$ (MW)  & -- & $10$ & $0$ & $10$ & $0$ \\
        $p^{dis,DA}_{k,t}$ (MW) & -- &  $0$ &  $0$ &  $0$ & $20$ \\
        $v^{DA}_{k,t}$ (MWh)    & $0$ & $10$ & $10$ & $20$ & $0$ \\
        \midrule
        $p^{ch,RT}_{k,t}$ (MW)  & -- & $10$ & $10^{\dagger}$ &  $0$ & $0$ \\
        $p^{dis,RT}_{k,t}$ (MW) & -- &  $0$ &  $0$ & $0$ & $20$  \\
        $v^{RT}_{k,t}$ (MWh)    & $0$ & $10$ & $20$ & $20$ & $0$ \\
        \midrule
        $\Delta p^{ch}_{k,t}$ (MW)  & -- & $0$ & $+10$ & $-10$ & $0$ \\
        $\Delta p^{dis}_{k,t}$ (MW) & -- & $0$ & $0$ & $0$ & $0$ \\
        \bottomrule
        \multicolumn{6}{l}{$\dagger$~Forced charge to comply with an exceptional dispatch. 
        }
    \end{tabular}
\end{table}

In this example, as described in Table \ref{tab:dispatch_concetual_example_2}, the storage resource is scheduled to charge in the first and third DA periods and discharge in the fourth. In the RT market, however, the storage is instructed to charge in the second period to comply with an exceptional dispatch. This exceptional dispatch ensures availability for a potential higher discharge from the third period onward, consequently forcing the resource to sell back energy in the third period. The day-ahead, real-time as well as total profits under the standard two settlement outcome and under the proposed rule are computed as follows.

\begin{align}
  \Pi^{DA}_k
    &= \sum_{t \in T} \lambda^{DA}_{t} (p^{dis,DA}_{k,t} - p^{ch,DA}_{k,t}) = \$2000
  \label{eq:ex_DA_profit_profit_concetual_example_2} \\
  \Pi^{RT}_k
    &= \sum_{t \in T} \lambda^{RT}_{t} (p^{dis,RT}_{k,t} - p^{ch,RT}_{k,t}) = \$3000
  \label{eq:ex_RT_profit_profit_concetual_example_2} \\ 
  \Pi^{RT,incr}_k
    &= \sum_{t \in T}
         \lambda^{RT}_t
         \bigl(\Delta p^{dis}_{k,t} - \Delta p^{ch}_{k,t}\bigr) = -\$2200
  \label{eq:RT_incr_formula_profit_concetual_example_2} \\
  TP^{St}_k &= \Pi^{DA}_k + \Pi^{RT,incr}_k = -\$200 \label{eq:total_profit_standard_2-settlemnet_concetual_example_2} \\
  TP^{St,\text{new}}_k &= \max \Bigl(\Pi^{DA}_k + \Pi^{RT,incr}_k,\Pi^{RT\text{-}only}_k\Bigr) = \$3000 \label{eq:total_profit_max-rule_concetual_example_2}
\end{align}

These circumstances result in a mild negative profit according to the standard settlement rule that is avoided by the proposed rule.

\vspace{-0.3cm}
\subsection{Conceptual example 3: buy and sell back situation}
\label{sec:buy_and_sell_back_example_concetual_example_3}

This example has the same setup of the conceptual example 2, with DA and RT prices reported in Table \ref{tab:prices_concetual_example_3}. The difference here is that the higher available SOC 
is deployed in real time at period 3 as reported in Table \ref{tab:dispatch_concetual_example_3} forcing the storage participant to perform sell back and buy back actions.

\begin{table}[!bt]
    \caption{Conceptual example 3 -- day-ahead and real-time prices (\$/MWh)}
    \label{tab:prices_concetual_example_3}
        \centering
        \renewcommand{\arraystretch}{1.2}
        \begin{tabular}{lccccccc}
        \toprule
        & \multicolumn{4}{c}{\textbf{Period}} \\
        \cmidrule(lr){2-7}
        \textbf{Price} & $t=1$ & $t=2$ & $t=3$ & $t=4$ \\
        \midrule
        $\lambda^{DA}_t$ (\$/MWh) & {\color{white}0}$50$ & {\color{white}0}$70$ & {\color{white}0}$50$ & $150$  \\
        $\lambda^{RT}_t$ (\$/MWh) & $180$ & $320$ & $100$ & $400$ \\
        \bottomrule
    \end{tabular}
\end{table}

\begin{table}[!bt]
    \caption{Conceptual example 3 -- DA schedule, RT actual dispatch, and SOC trajectories}
    \label{tab:dispatch_concetual_example_3}
    \centering
    \begin{threeparttable}
    \renewcommand{\arraystretch}{1.1}
    \begin{tabular}{lccccc}
        \toprule
        \textbf{Quantity} & \textbf{Init.} & $t=1$ & $t=2$ & $t=3$ & $t=4$ \\
        \midrule
        $p^{ch,DA}_{k,t}$ (MW)  & -- & $10$ & $0$ & $10$ & $0$ \\
        $p^{dis,DA}_{k,t}$ (MW) & -- &  $0$ &  $0$ &  $0$ & $20$ \\
        $v^{DA}_{k,t}$ (MWh)    & $0$ & $10$ & $10$ & $20$ & $0$ \\
        \midrule
        $p^{ch,RT}_{k,t}$ (MW)  & -- & $10$ & $10^{\dagger}$ &  $0$ & $0$ \\
        $p^{dis,RT}_{k,t}$ (MW) & -- &  $0$ &  $0$ & $20^{\ddagger}$ & $0$  \\
        $v^{RT}_{k,t}$ (MWh)    & $0$ & $10$ & $20$ & $0$ & $0$ \\
        \midrule
        $\Delta p^{ch}_{k,t}$ (MW)  & -- & $0$ & $+10$ & $-10$ & $0$ \\
        $\Delta p^{dis}_{k,t}$ (MW) & -- & $0$ & $0$ & $+20$ & $-20$ \\
        \bottomrule
    \end{tabular}
    \begin{tablenotes}
        \small
        \item[$\dagger$] Forced charge to comply with an exceptional dispatch.
        \item[$\ddagger$] Deployment of the higher SOC.
    \end{tablenotes}
    \end{threeparttable}
\end{table}

The day-ahead, real-time as well as total profits under the standard two settlement outcome and under the proposed rule are computed as follows.

\begin{align}
  \Pi^{DA}_k
    &= \sum_{t \in T} \lambda^{DA}_{t} (p^{dis,DA}_{k,t} - p^{ch,DA}_{k,t}) = \$2000
  \label{eq:ex_DA_profit_profit_concetual_example_3} \\
  \Pi^{RT}_k
    &= \sum_{t \in T} \lambda^{RT}_{t} (p^{dis,RT}_{k,t} - p^{ch,RT}_{k,t}) = -\$3000
  \label{eq:ex_RT_profit_profit_concetual_example_3} \\ 
  \Pi^{RT,incr}_k
    &= \sum_{t \in T}
         \lambda^{RT}_t
         \bigl(\Delta p^{dis}_{k,t} - \Delta p^{ch}_{k,t}\bigr) = -\$8200
  \label{eq:RT_incr_formula_profit_concetual_example_3} \\
  TP^{St}_k &= \Pi^{DA}_k + \Pi^{RT,incr}_k = -\$6200 \label{eq:total_profit_standard_2-settlemnet_concetual_example_3} \\
  TP^{St,\text{new}}_k &= \max \Bigl(\Pi^{DA}_k + \Pi^{RT,incr}_k,\Pi^{RT\text{-}only}_k\Bigr) \notag \\
  & \hspace{140pt} = -\$3000 \label{eq:total_profit_max-rule_concetual_example_3}
\end{align}

In this situation, the resulting profit is negative under both standard and proposed settlement rules. Nonetheless, the proposed rule alleviates the incurred negative profit.


\begin{figure}[tb]
    \centering
    \includegraphics[width=0.50\textwidth,height=.5\textheight,keepaspectratio]{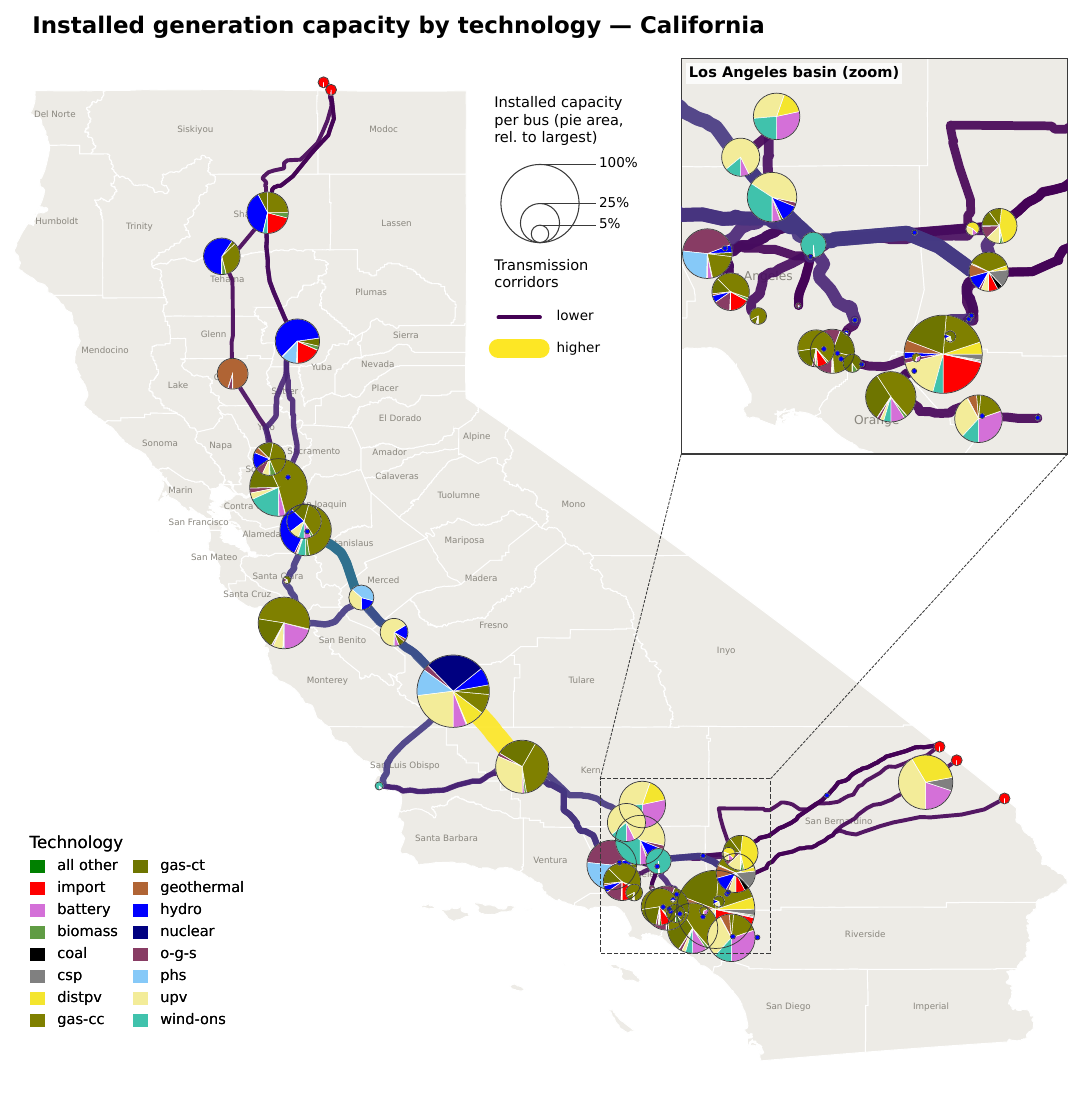}
    \caption{California 500 kV power system.}
    \label{fig:California500kVSystem}
\end{figure}

\begin{figure}[tb]
    \centering
    \includegraphics[width=0.50\textwidth,height=.5\textheight,keepaspectratio]{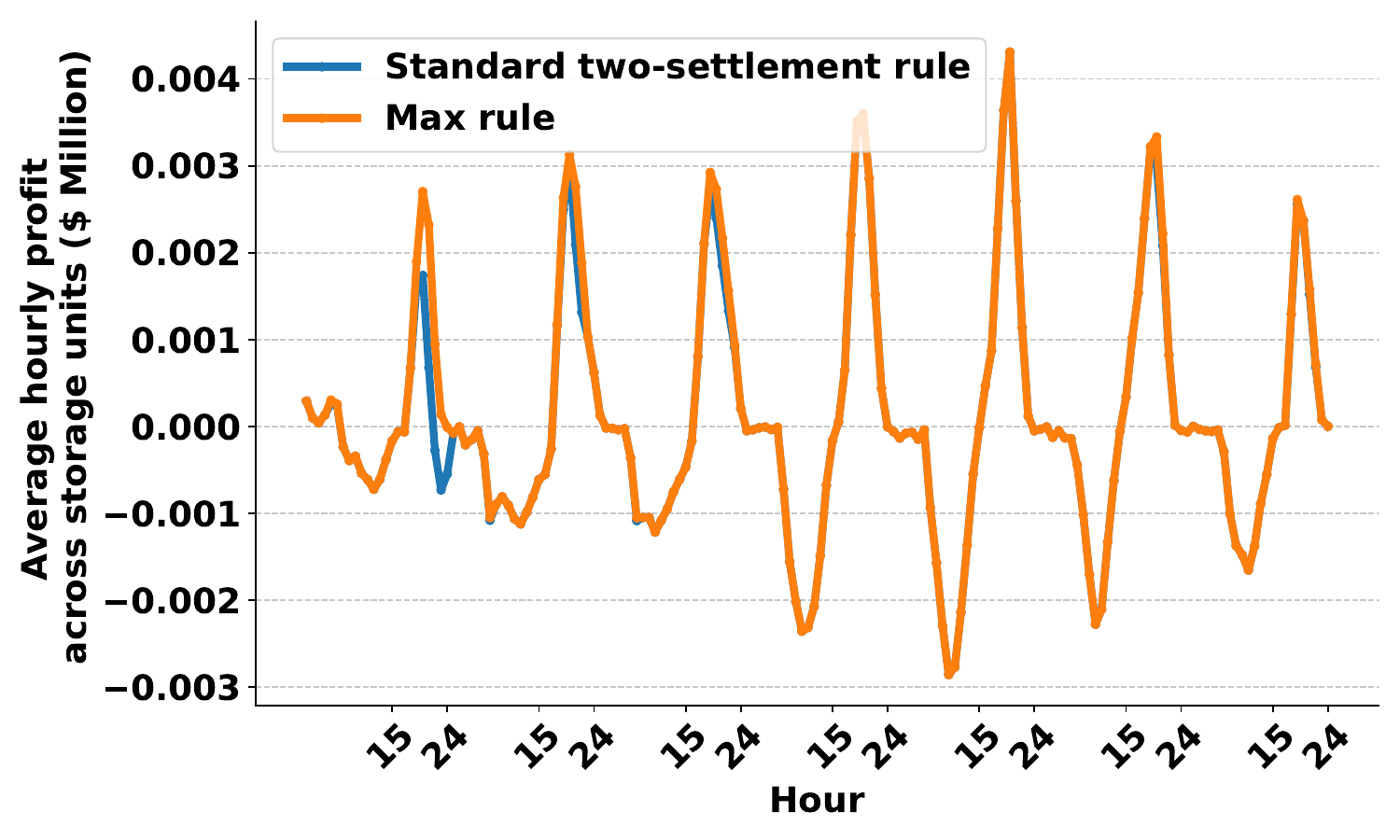}
    \caption{Storage units' average hourly profit for all hours of the normal week.}
    \label{fig:storage_hourly_profit_line_normal_week}
\end{figure}

\begin{figure}[tb]
    \centering
    \includegraphics[width=0.50\textwidth,height=.5\textheight,keepaspectratio]{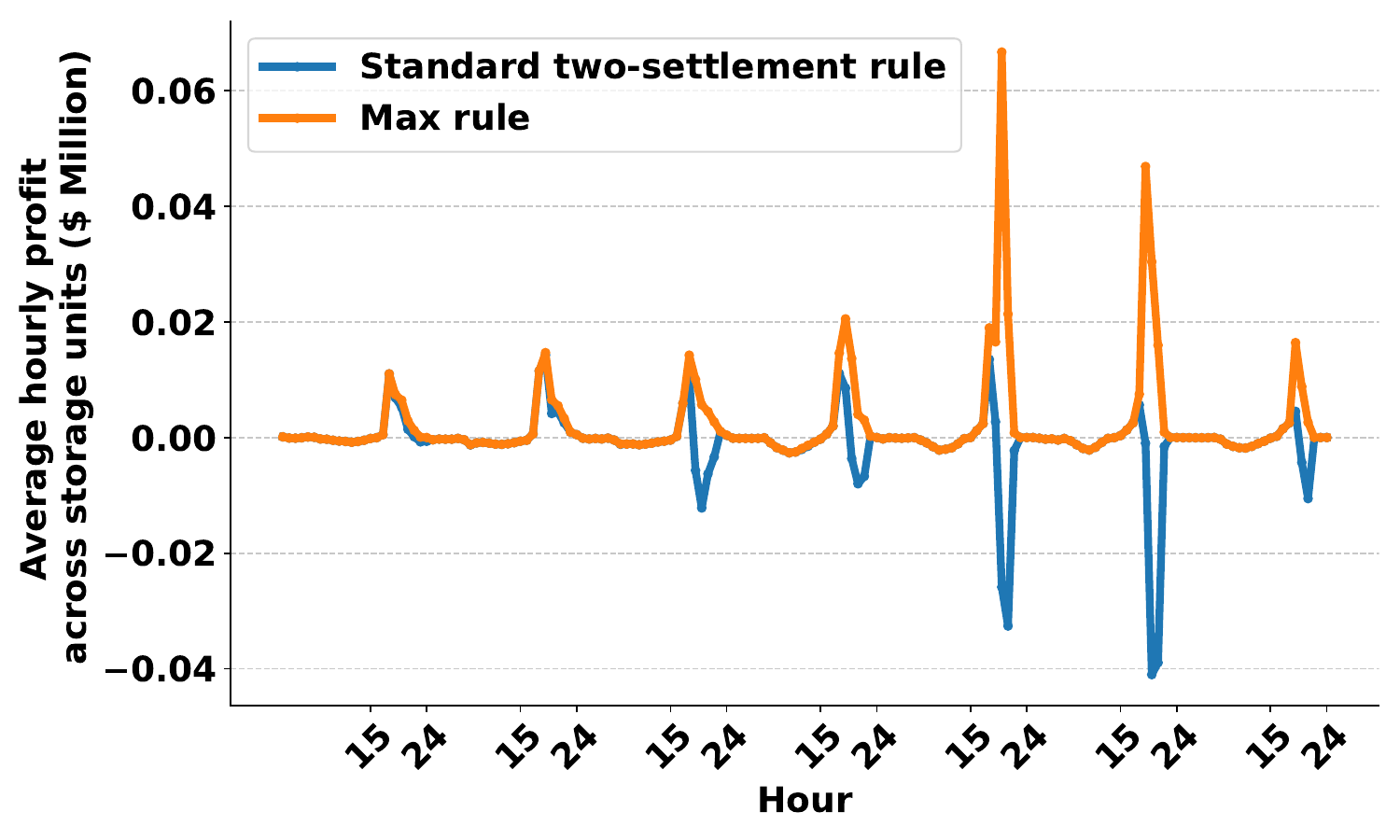}
    \caption{Storage units' average hourly profit for all hours of the critical week.}
    \label{fig:storage_hourly_profit_line_critical_week}
\end{figure}

\begin{figure}[tb]
    \centering
    \includegraphics[width=0.50\textwidth,height=.5\textheight,keepaspectratio]{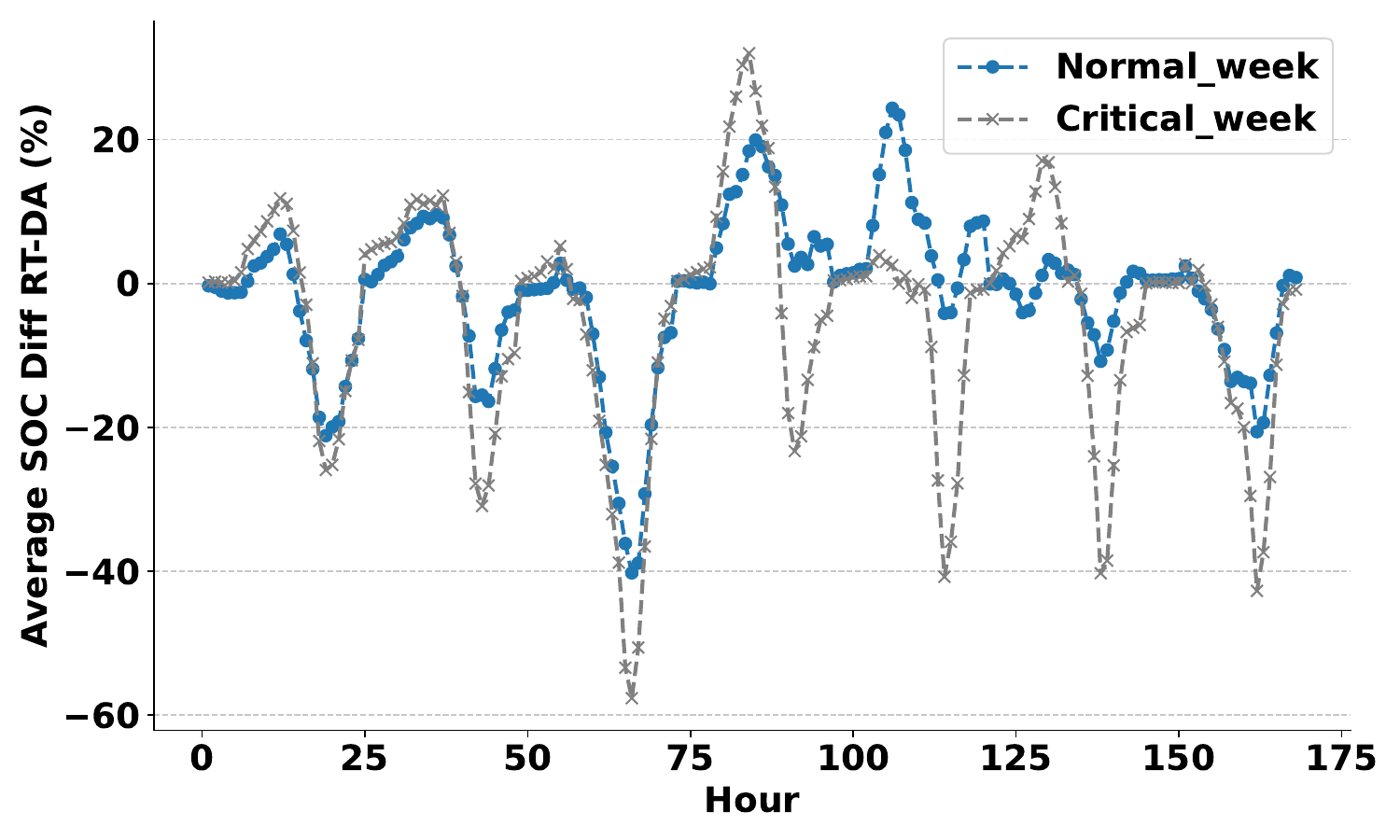}
    \caption{Average SOC differences across storage units.}
    \label{fig:average_SOC_differences}
\end{figure}

\begin{figure}[tb]
    \centering
    \includegraphics[width=0.50\textwidth,height=.5\textheight,keepaspectratio]{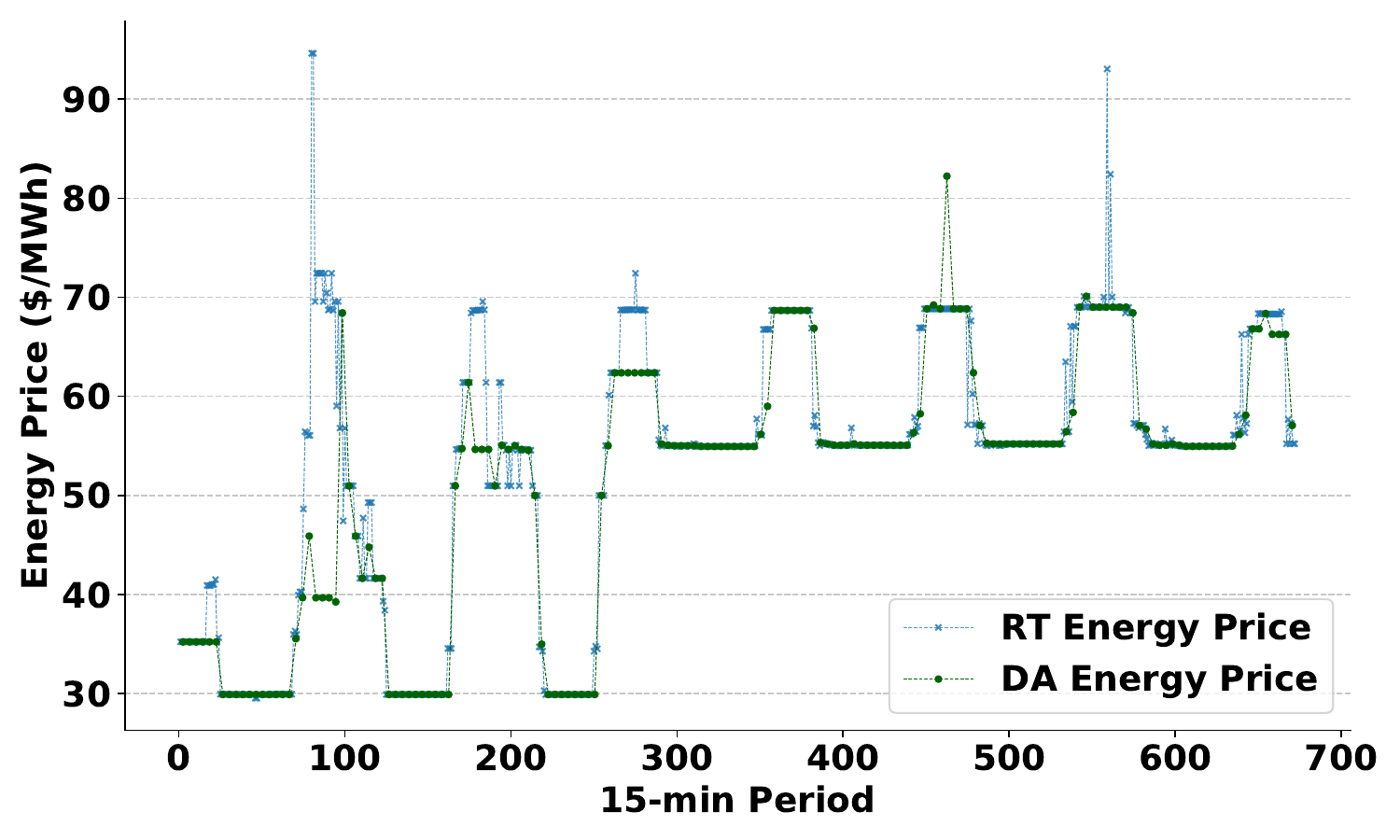}
    \caption{Average energy prices across the different nodes of the system for all hours of the normal week.}
    \label{fig:average_energy_prices_normal_week}
\end{figure}

\begin{figure}[tb]
    \centering
    \includegraphics[width=0.50\textwidth,height=.5\textheight,keepaspectratio]{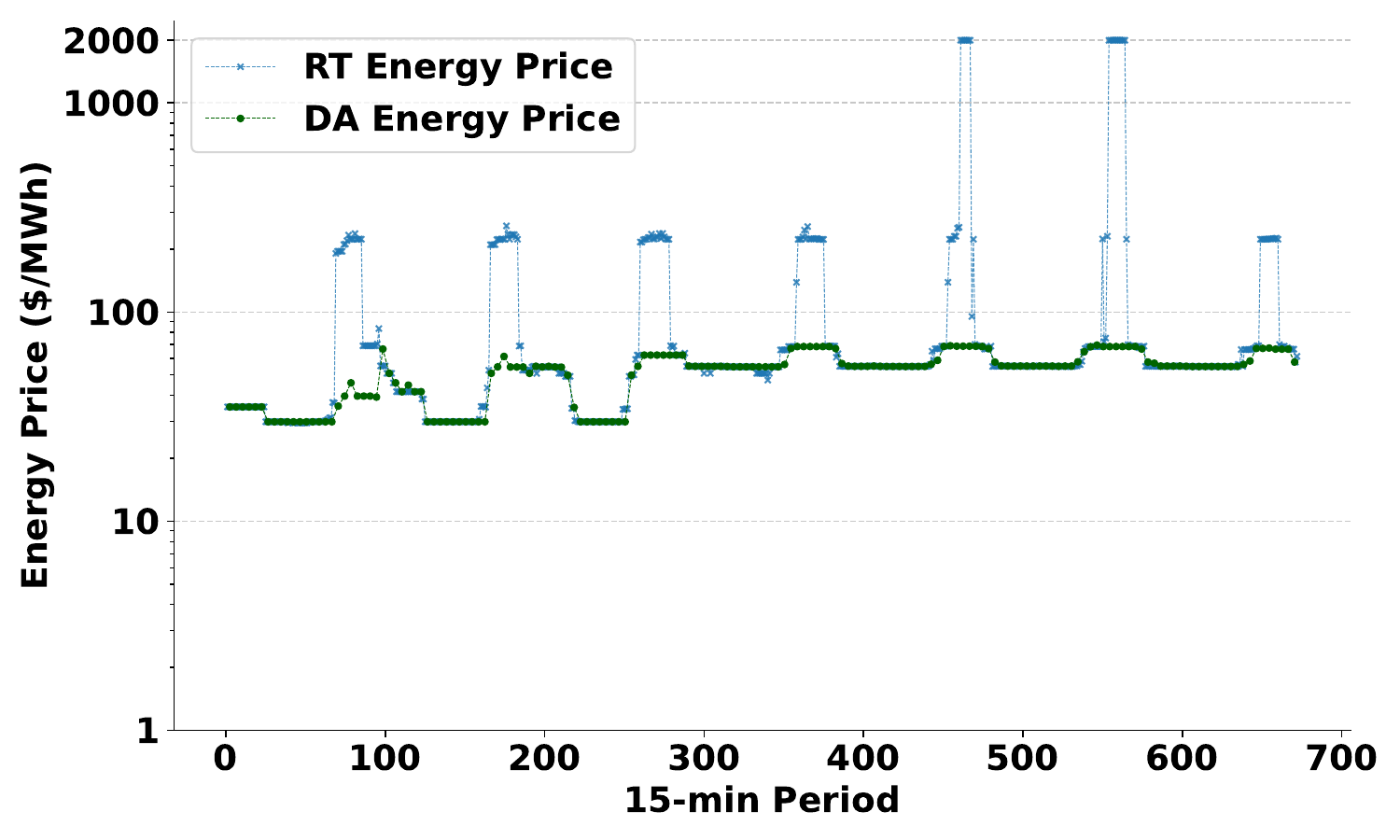}
    \caption{Average energy prices across the different nodes of the system for all hours of the critical week.}
    \label{fig:average_energy_prices_critical_week}
\end{figure}

\begin{figure}[tb]
    \centering
    \includegraphics[width=0.50\textwidth,height=.5\textheight,keepaspectratio]{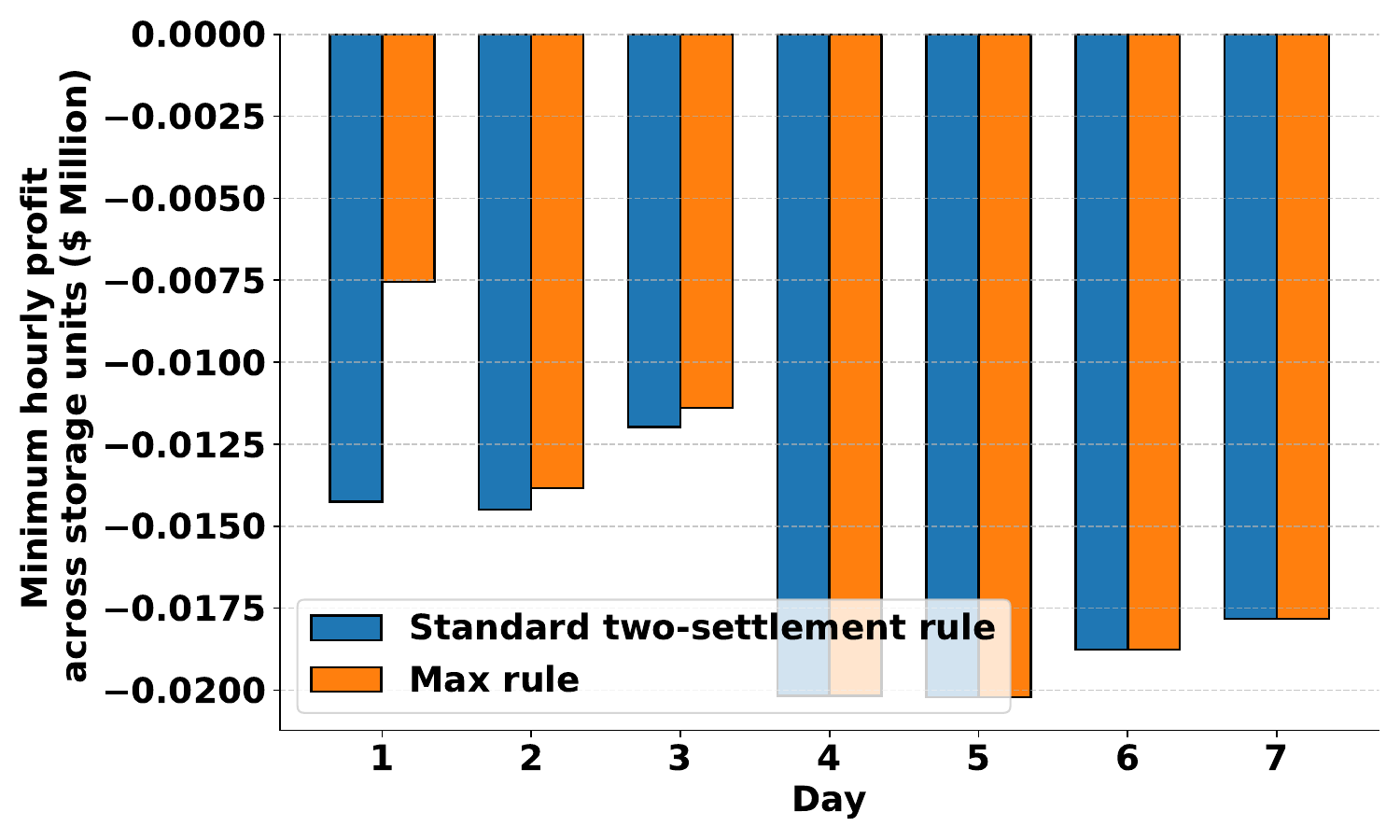}
    \caption{Storage units' minimum hourly profit for each day under normal week.}
    \label{fig:storage_min_profit_comparison_normal_week}
\end{figure}

\begin{figure}[tb]
    \centering
    \includegraphics[width=0.50\textwidth,height=.5\textheight,keepaspectratio]{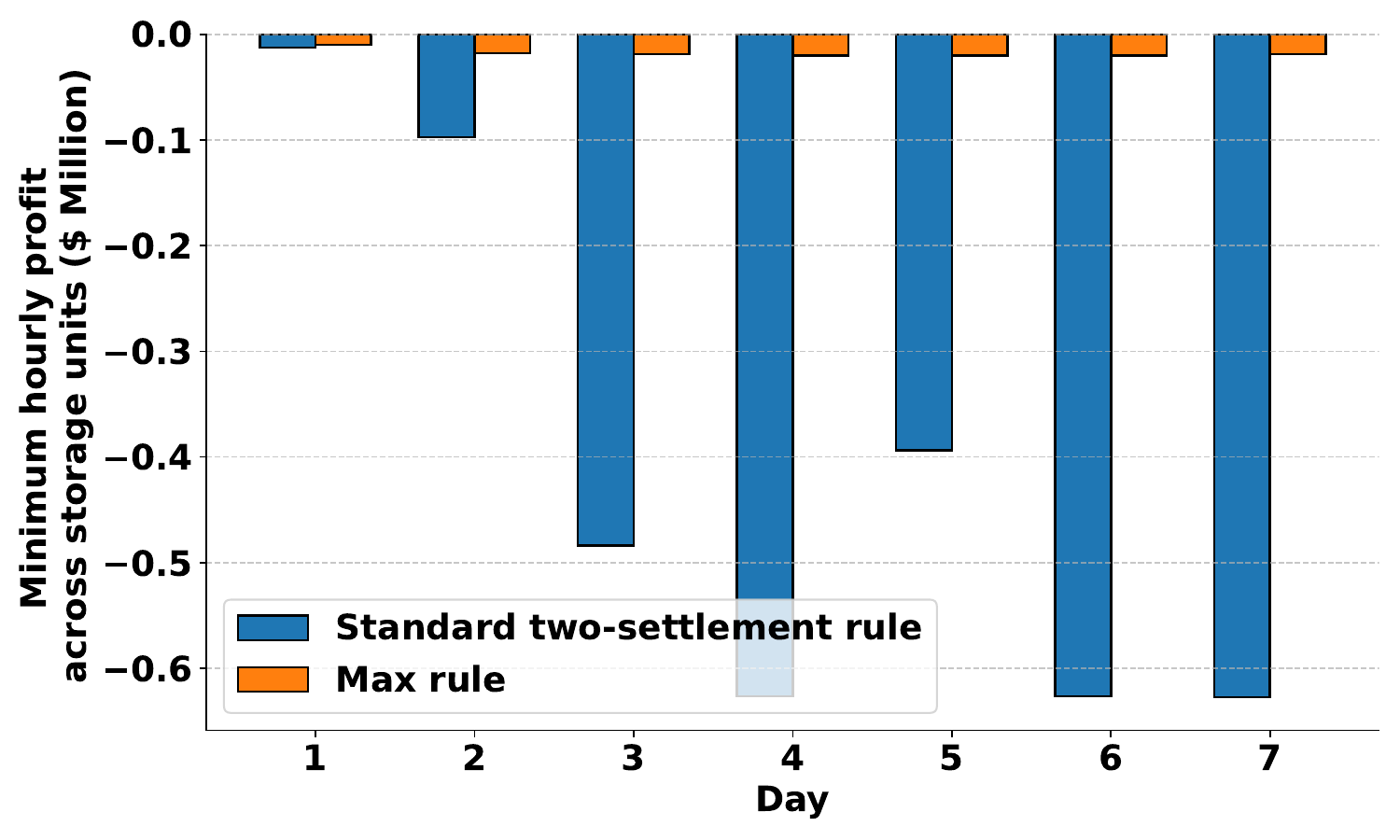}
    \caption{Storage units' minimum hourly profit for each day under critical week.}
    \label{fig:storage_min_profit_comparison_critical_week}
\end{figure}

\begin{figure}[tb]
    \centering
    \includegraphics[width=0.50\textwidth,height=.5\textheight,keepaspectratio]{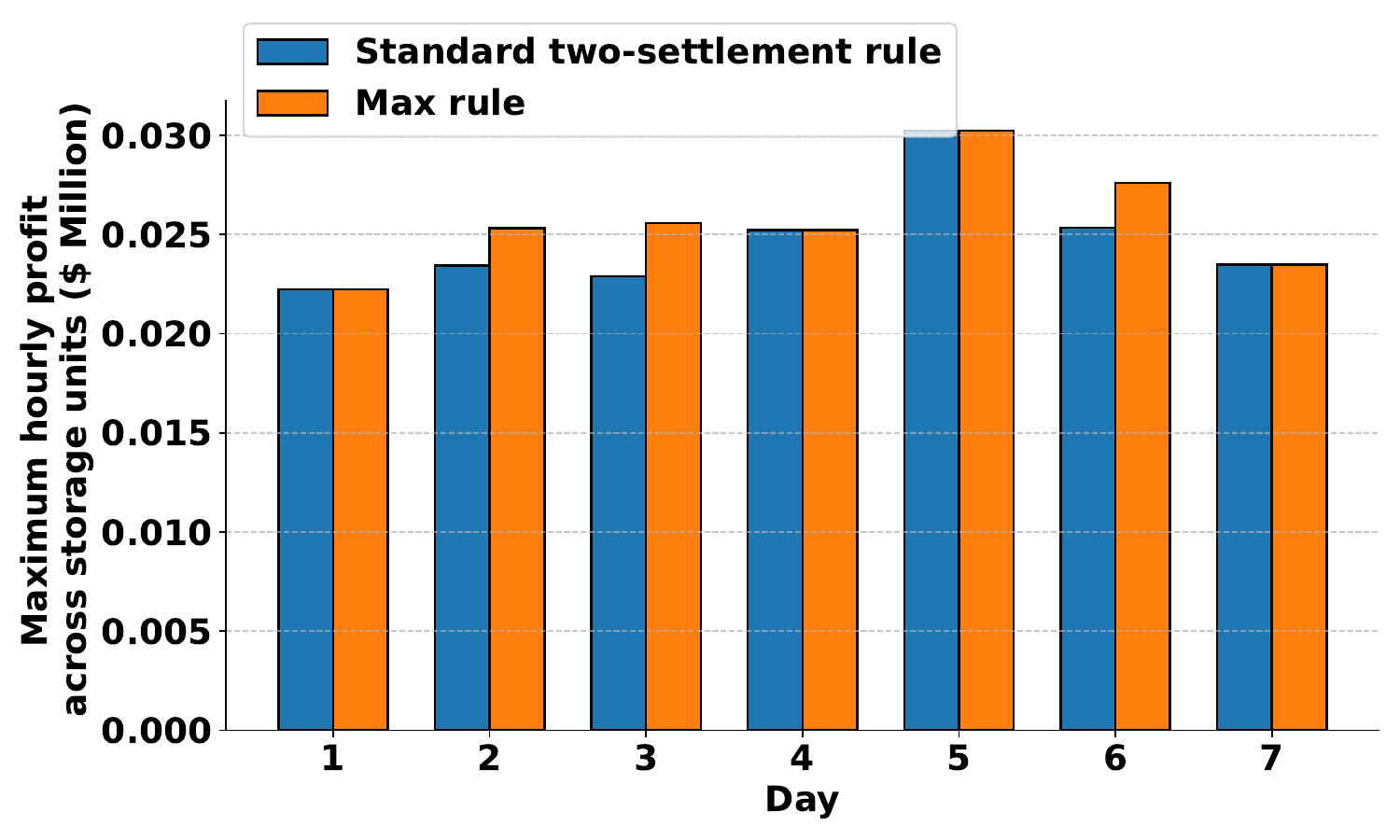}
    \caption{Storage units' maximum hourly profit for each day under normal week.}
    \label{fig:storage_max_profit_comparison_normal_week}
\end{figure}

\begin{figure}[tb]
    \centering
    \includegraphics[width=0.50\textwidth,height=.5\textheight,keepaspectratio]{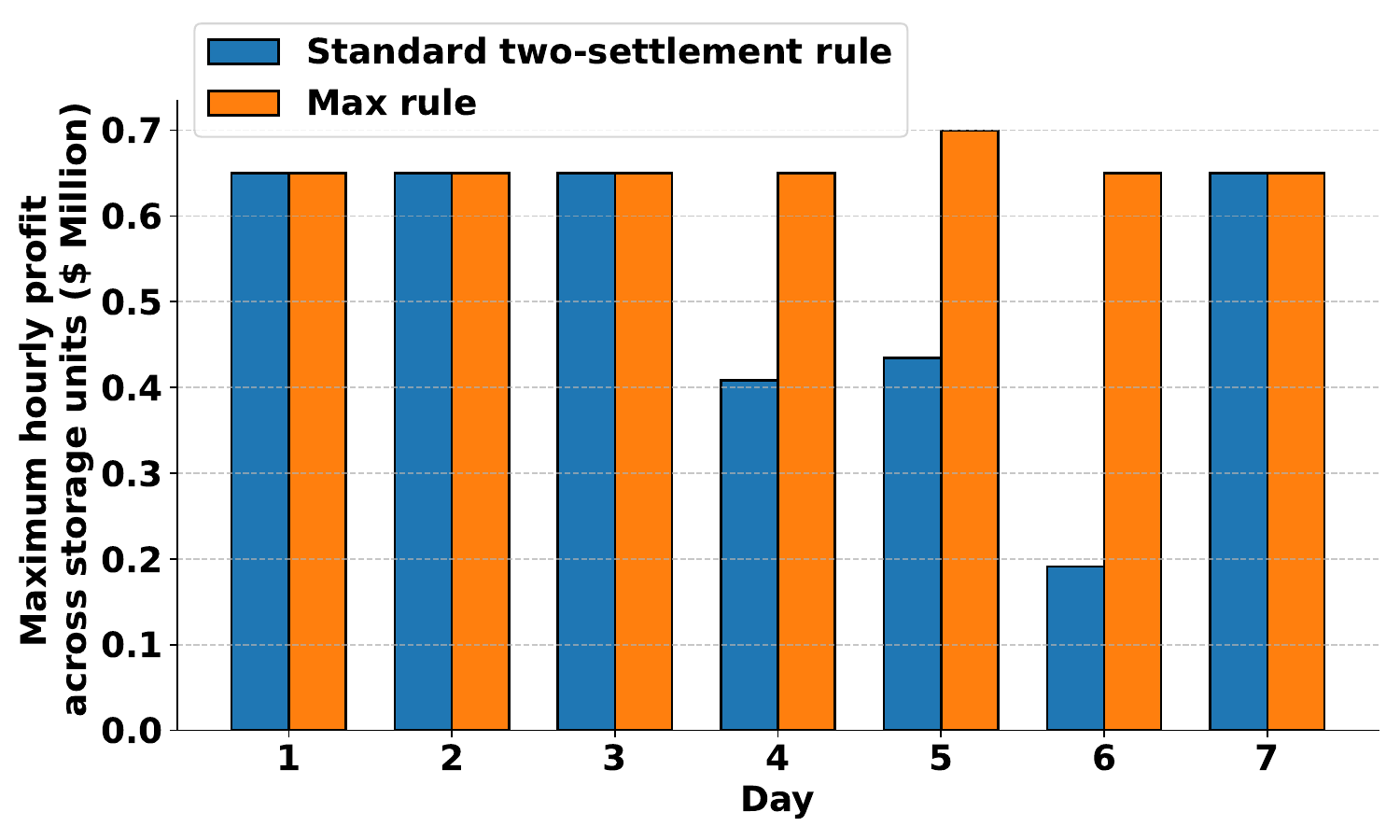}
    \caption{Storage units' maximum hourly profit for each day under critical week.}
    \label{fig:storage_max_profit_comparison_critical_week}
\end{figure}
\vspace{-0.4cm}
\section{Case study}
\label{sec:case_study}

\begin{figure}[tb]
    \centering
    \includegraphics[width=0.50\textwidth,height=.5\textheight,keepaspectratio]{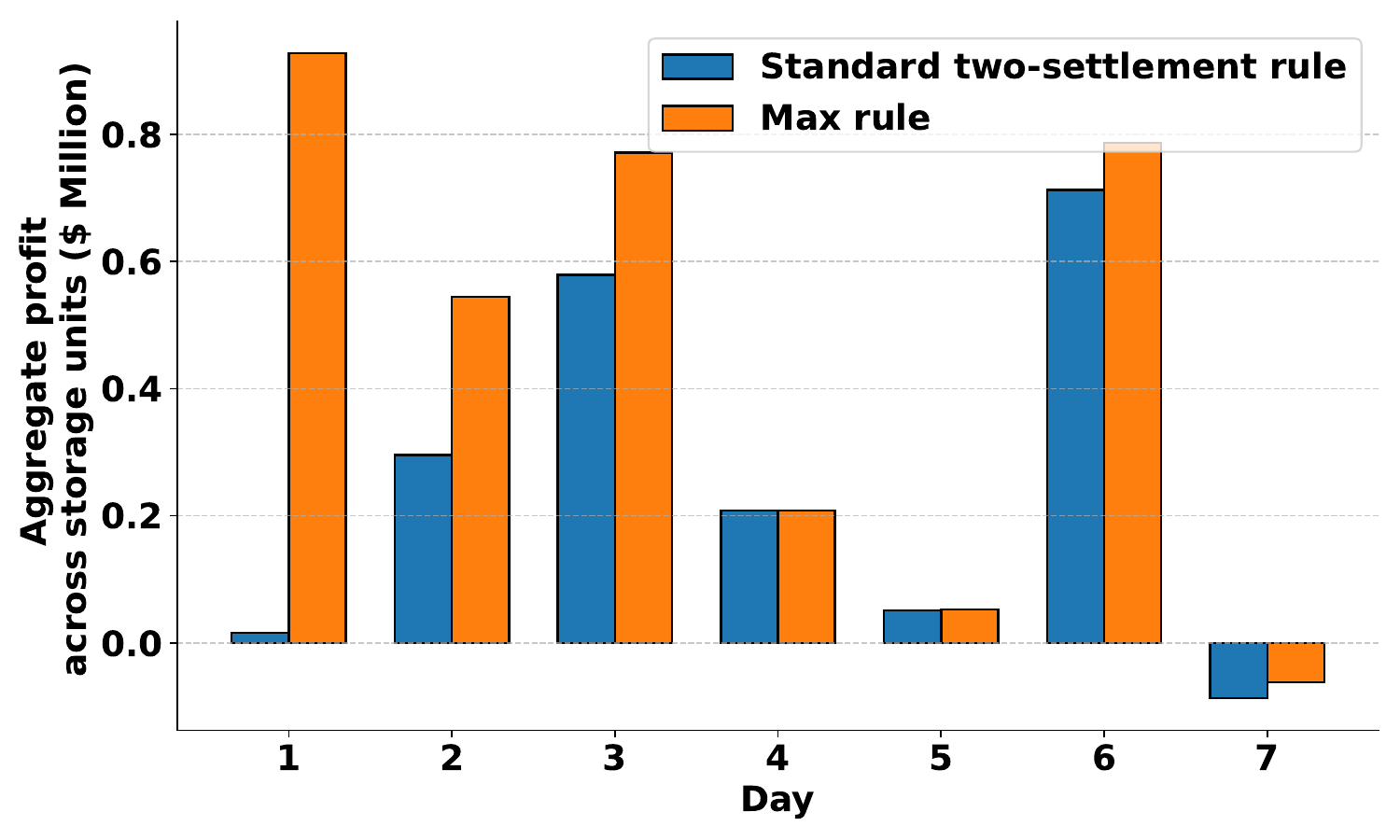}
    \caption{Storage units' daily profits under normal week.}
    \label{fig:storage_profit_daily_comparison_normal_week}
\end{figure}

\begin{figure}[tb]
    \centering
    \includegraphics[width=0.50\textwidth,height=.5\textheight,keepaspectratio]{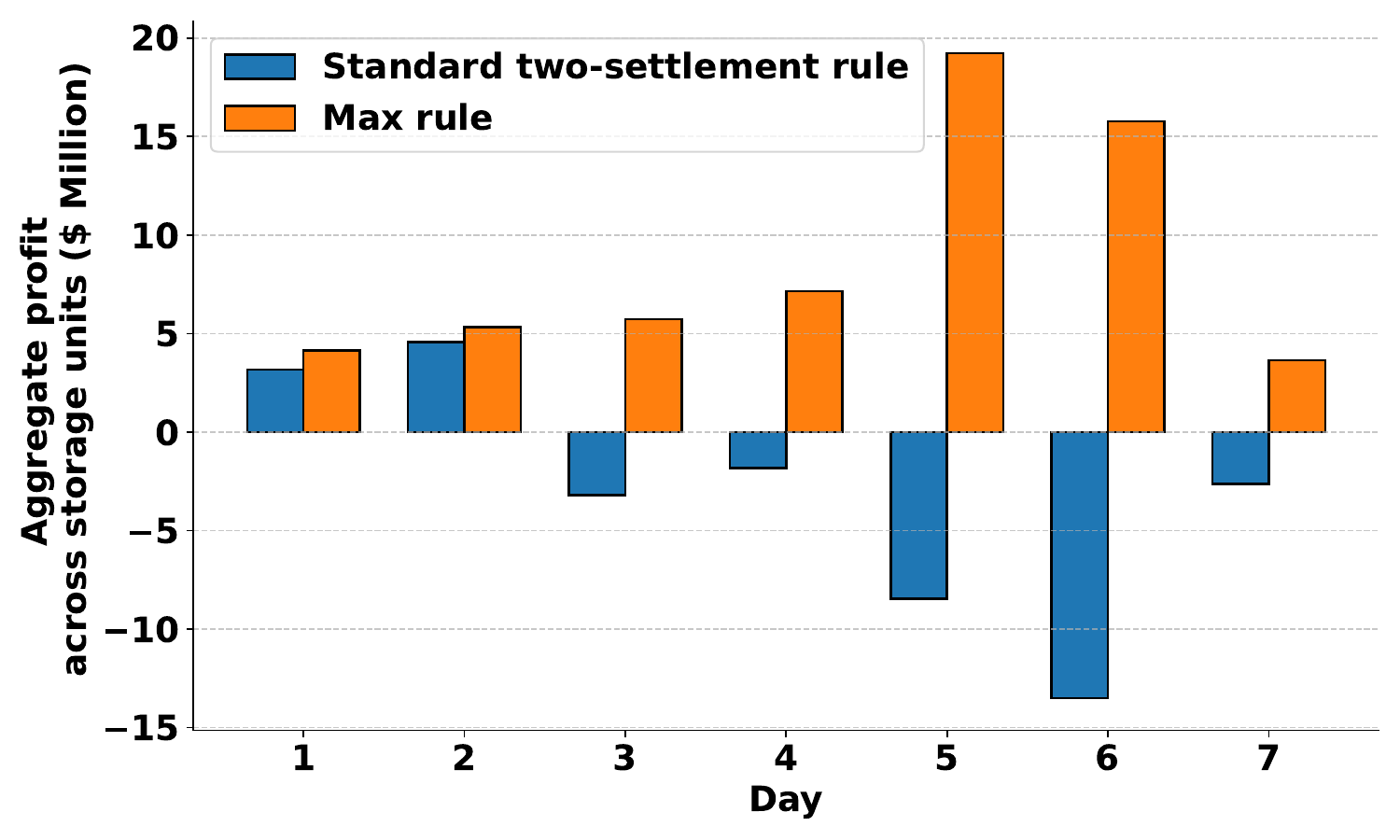}
    \caption{Storage units' daily profits under critical week.}
    \label{fig:storage_profit_daily_comparison_critical_week}
\end{figure}

\begin{figure}[!t]
    \centering
    \includegraphics[width=0.50\textwidth,height=.5\textheight,keepaspectratio]{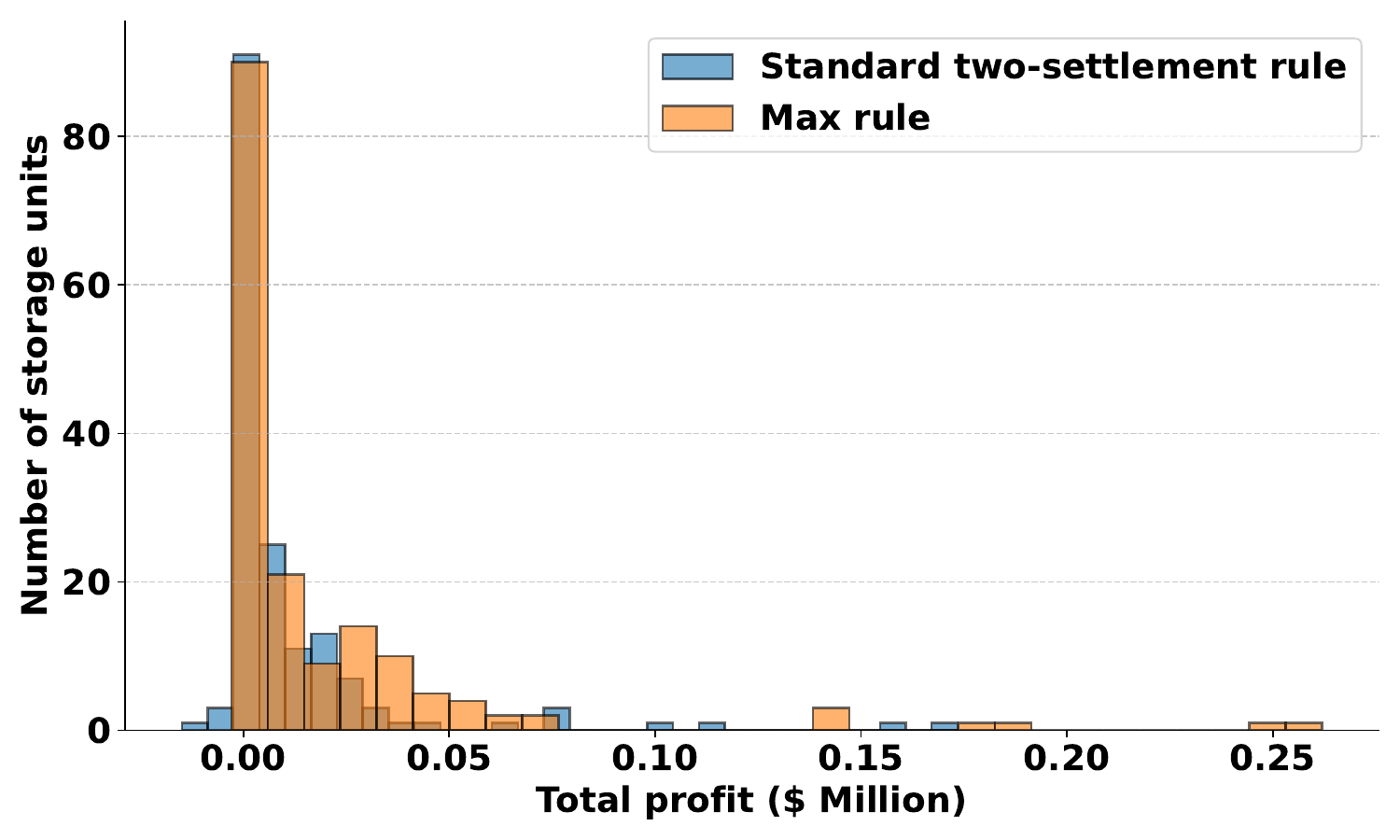}
    \caption{Storage units' profit histogram under normal week.}
    \label{fig:storage_individual_profit_distribution_normal_week}
\end{figure}

\begin{figure}[!t]
    \centering
    \includegraphics[width=0.50\textwidth,height=.5\textheight,keepaspectratio]{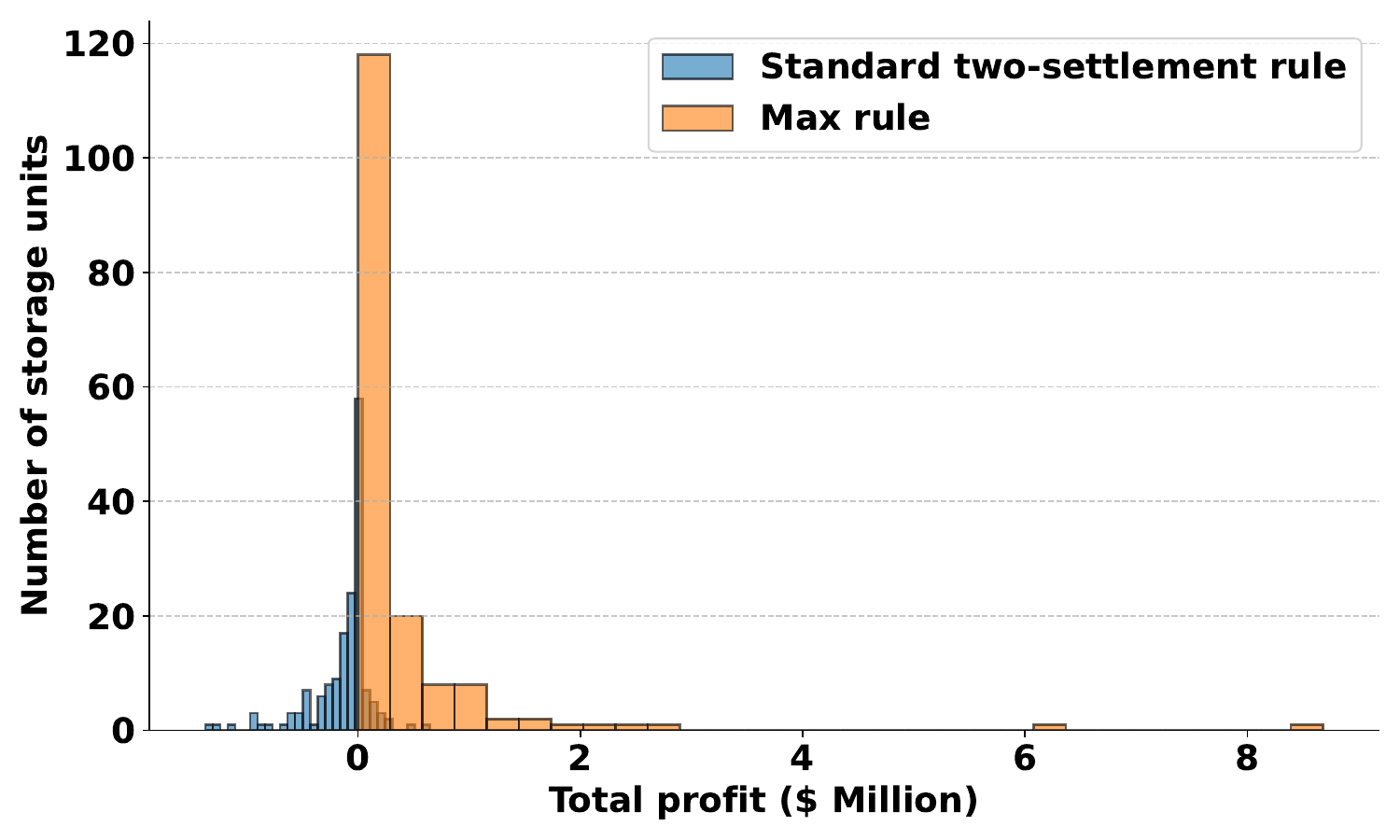}
    \caption{Storage units' profit histogram under critical week.}
    \label{fig:storage_individual_profit_distribution_critical_week}
\end{figure}

We illustrate the economic performance of storage assets under the standard two-settlement and the max rules in this case study, providing a comparison across several dimensions. To do so, we use a reduced version of the California Test System \cite{Taylor2024_CATSystem}, considering only the 500 kV transmission grid as depicted in Fig. \ref{fig:California500kVSystem}. We have updated this reduced system with (i) 2023 installed generation and storage data from the United States Energy Information Area (EIA) \cite{US_EIA_installed_cap} and (ii) 2023 demand information from CAISO (Production and Curtailments Data in \cite{CAISO_DemandData}).  Our full datasets can be found in \cite{CaseStudyRepository_2-SettlementRulePaper}. With this input data and selecting information from August 2023, we run two full week experiments for each type of compensation rule under consideration. The first one is labeled as {\it normal week} and has an average daily peak netload variation from DA to RT of 2.42\%, while the second one, hereinafter referred to as {\it critical week} has an average daily peak netload variation of 16.75\%. 
Each full week experiment consists of optimizing DA and RT market clearings for seven days. Each DA problem considers 24 hourly binding periods and other 24 look-ahead periods per day whereas each RT problem considers 1 binding and 3 advisory (look-ahead) 15-min periods, totaling 96 RT periods per day. 


The temporal evolution of the average hourly profits obtained by storage systems is captured in Figs. \ref{fig:storage_hourly_profit_line_normal_week} and \ref{fig:storage_hourly_profit_line_critical_week}, which exemplify how the standard two-settlement and the max rules can yield a similar average profit for storage units under a normal week and significantly different average profit under a critical week. Essentially, under a normal week, when the variation between DA and RT netloads is within an ordinary range, the SOC of storage units in RT is on average 20\% lower than in DA three times during the whole week compared to seven times during a critical week, with this difference also reaching -40\% three times and almost -60\% once as depicted in \ref{fig:average_SOC_differences}.    

In addition to having lower RT versus DA average SOC more often, the critical week leads to more events of higher energy prices in RT as illustrated in Figs. \ref{fig:average_energy_prices_normal_week} and \ref{fig:average_energy_prices_critical_week}. These higher price differences combined with higher SOC differences exacerbate the differences between the empirical bounds for hourly profits computed in this exercise for the two settlement rules. Figs. \ref{fig:storage_min_profit_comparison_normal_week}, \ref{fig:storage_min_profit_comparison_critical_week}, \ref{fig:storage_max_profit_comparison_normal_week}, and \ref{fig:storage_max_profit_comparison_critical_week} depict minimum and maximum hourly profits for each day across storage units under the two rules.  While the minimum and maximum values are relatively close for the two rules under a normal week, a substantial difference arises under the critical week. Under the standard rule, the most severely impacted storage units lose over \$600k during a single hour on three out of the seven days. 


The overall daily financial impact is summarized in Figs. \ref{fig:storage_profit_daily_comparison_normal_week} and \ref{fig:storage_profit_daily_comparison_critical_week}. In both scenarios, the max rule consistently provides higher aggregate compensation than the standard two-settlement rule. However, the gap between the two rules is significantly more pronounced in Fig. \ref{fig:storage_profit_daily_comparison_critical_week}, indicating that the max rule provides a critical hedge against the price volatility characteristic of a system-stressed week.


Finally, the distribution of total profits for the whole week across individual storage assets is presented in Figs. \ref{fig:storage_individual_profit_distribution_normal_week} and \ref{fig:storage_individual_profit_distribution_critical_week}. The transition from the standard rule to the max rule results in a rightward shift of the profit distribution. While the normal week shows a relatively tight distribution for both rules, the critical week's comparison illustrates a clearer distinction where the max rule profit distribution is dispersed over positive profits and the standard rule leads to several negative total profits.  

%
\vspace{-0.3cm}
\section{Conclusion}
\label{sec:conclusion}

This paper discussed a fundamental asymmetry in the compensation rule applied to different types of resources in
two-settlement electricity markets. Using linear-programming
strong duality, we established two structural results.

First, conventional generators are \emph{SOC-unlimited} resources.
Therefore, their ability to deliver a day-ahead awarded schedule in real time
is bounded only by their nameplate capacity, which is identical in
both settlements. As a consequence, the combined DA+RT profit is not going to be negative due to lack of capacity in RT.

Second, storage systems are \emph{SOC-limited} resources. Their
combined DA+RT profit depends explicitly on the gap between the SOC
available at the DA clearing stage and the SOC realized at the RT
stage. We showed formally that this gap 
can drive the combined profit of a
storage participant to negative values.

From a broader market-design perspective, the results of this
paper underscored a general principle that compensation rules
designed for SOC-unlimited resources will systematically
produce unintended outcomes when applied to SOC-limited
resources, because the two-settlement profit of the latter
depends on a state variable---the SOC---that can be altered
by market-design mechanisms between the two clearings.
The proposed max rule operationalizes
this insight in a single, transparent formula that requires no architectural change to the two-settlement structure.
\vspace{-0.8cm}
{\color{black}
\section*{Acknowledgment}

The author would like to thank Dr. Fan Zhang (Principal Engineer in the Power System Technology Development Group at the California ISO) for the insightful discussions and valuable suggestions that motivated this paper, and Dr. Miguel Heleno (Staff Scientist in the Energy Technologies Area at Lawrence Berkeley National Laboratory) for his fruitful feedback. This work was funded by the U.S. Department of Energy, Office of Electricity, under contract DE-AC02-05CH11231.
}
\vspace{-0.5cm}
\bibliographystyle{IEEEtran}
\bibliography{References}


\end{document}